\documentclass[11pt,a4paper]{amsart}
\usepackage[margin=2.5cm]{geometry}
\usepackage{graphicx}
\usepackage{epstopdf}
\usepackage{bm} 
\usepackage{soul}
\usepackage{url}
\usepackage[dvipsnames]{xcolor}
\usepackage{pinlabel}
\makeatletter
\def\blfootnote{\xdef\@thefnmark{}\@footnotetext}
\makeatother
\usepackage[style=numeric, sorting=nyt]{biblatex} 
\usepackage{fourier}\usepackage{microtype}
\usepackage{caption}\usepackage{verbatim}
\renewcommand\footnotemark{}

\newcommand{\N}{\mathbb N}

\newcommand{\CP}{\mathbb{CP}}

\newcommand{\Z}{\mathbb Z} 
\newcommand{\F}{\mathcal F}
\newcommand{\TT}{\mathbb T}

\newcommand{\fS}{\mathfrak{S}}

\newcommand{\de}{\delta}
\newcommand{\ep}{\varepsilon}
\newcommand{\si}{\sigma}
\newcommand{\ga}{\gamma}

\newcommand{\Si}{\Sigma}

\newcommand{\la}{\lambda}

\newcommand{\x}{\times}
 
\newcommand{\hra}{\hookrightarrow}
\newcommand{\del}{\partial}

\renewcommand{\H}{\mathcal H}
\renewcommand{\int}{\rm int}
\newcommand{\co}{\thinspace\colon}

\DeclareMathOperator{\sgn}{sgn}
\DeclareMathOperator{\fr}{fr}
\DeclareMathOperator{\Map}{Map}

\newcommand{\tr}{tr}
\newcommand{\lk}{lk}
\newcommand{\bde}{\bm{\de}}
\newcommand{\bx}{\bm{x}}

\newtheorem{thm}{Theorem}[section]
\newtheorem{lemma}[thm]{Lemma}
\newtheorem{cor}[thm]{Corollary}
\newtheorem{prop}[thm]{Proposition}

\theoremstyle{definition}

\newtheorem{defn}[thm]{Definition}
\newtheorem{rmk}[thm]{Remark}

\newtheorem*{ack}{Acknowledgments}

\numberwithin{equation}{section}

\newcommand{\Addresses}{{
		\bigskip
		\footnotesize
		
		\textsc{Dipartimento di Matematica, Largo Bruno Pontecorvo 5, 56127 Pisa, Italy}\par\nopagebreak
		\textit{E-mail addresses:}\ \texttt{paolo.lisca@unipi.it, andrea.parma94@gmail.com}
}}

\begin{document}
\title{Horizontal decompositions, II}
\author{Paolo Lisca \and Andrea Parma}
\date{\today}
\maketitle	

\begin{abstract}
We complete the classification of the smooth, closed, oriented $4$-manifolds having Euler characteristic less than four and a horizontal handlebody decomposition of genus one. 
We use the classification result to find a large family of rational homology ball smoothings of cyclic quotient singularities which can be smoothly embedded into the complex projective plane. 
Our family contains all such rational balls previously known to embed into $\CP^2$ and infinitely many more.
We also show that a rational ball of our family admits an almost-complex embedding in $\CP^2$ if and only if it admits a symplectic embedding.
\end{abstract}

\blfootnote{2020 {\it Mathematics Subject Classification.}\, 57R40 (Primary), 57K43, 57R17 (Secondary).} 
\blfootnote{Keywords: rational homology balls, smooth embeddings, handlebody decompositions.}

\section{Introduction}\label{s:intro} 

Let $B_{p,q}$, where $p>q\geq 0$ are coprime integers, be the rational homology ball smoothing of the cyclic quotient singularity $\frac1{p^2}(pq-1,1)$. 
The rational balls $B_{p,q}$ were used in the smooth rational blow-down construction~\cite{FS97} and in its symplectic counterpart~\cite{Sy98}. 
They played an important role in the construction of exotic $4$-manifolds, starting with the papers~\cite{Pa05, FS06}. 
The problem of embedding the $B_{p,q}$'s, smoothly or symplectically, in a $4$-manifold was considered in~\cite{Kh12, Kh13, Kh14}. 
As pointed out in~\cite[Section~2]{ES18}, the existence of holomorphic embeddings $B_{p,q}\hra\CP^2$ for certain $(p,q)$'s follows from results in algebraic geometry~\cite{HP10}, and explicit symplectic embeddings can be constructed. 
Moreover, Evans and Smith~\cite{ES18} showed that no further symplectic embeddings exist beyond those which are already known.  
Owens~\cite{Ow20} showed that infinitely many $B_{p,q}$'s admit smooth embeddings in 
$\CP^2$ but by~\cite{ES18} they cannot be symplectically embedded. 
The authors extended Owens' family~\cite{LP22-1} and proved the non-existence of almost 
complex embeddings~\cite{LP20} without relying on~\cite{ES18}. The present paper is a natural 
continuation of~\cite{LP22-2}, where we introduced certain handlebody decompositions that 
we call horizontal, classified the closed $4$-manifolds with the simplest horizontal 
decompositions and in doing so we recovered infinitely many of the known smooth embeddings $B_{p,q}\hra\CP^2$.

This paper consists of two parts. In the first part we complete the classification of the 
smooth, closed, oriented $4$-manifolds $X$ with Euler characteristic $\chi(X)\leq 3$ and a 
horizontal decomposition of genus $1$. In the second part we exploit the classification 
result to determine the rational balls $B_{p,q}$ which admit smooth embeddings 
in $\CP^2$ induced by a genus-$1$ horizontal decomposition. It turns out that in this 
way we are able to find the largest family so far of $B_{p,q}$'s smoothly 
embedded into the complex projective plane. The family contains all such rational balls 
previously known to embed into $\CP^2$ and many more. We show that each rational ball of the family 
admits a symplectic embedding in $\CP^2$ if and only if it admits an almost-complex 
embedding. 

In order to state our results we need to recall the definition of a horizontal decomposition~\cite{LP22-2}. 
Let $\H$ be a handlebody decomposition of a smooth, oriented 4-dimensional cobordism $X\co\del_- X\to\del_+ X$. 
Suppose that $\H$ has $h_i$ handles of index $i$ and let $\del_\pm^{h_i} X = \del_\pm X\#^{h_i} S^1\x S^2$ for $i\in\{1,3\}$. 
Let $X_L\co\del_-^{h_1} X\to \del_+^{h_3} X$ be the cobordism obtained by attaching the $2$-handles along a framed link $L = \bigcup_{i=1}^{h_2} \ga_i\subset \del_-^{h_1} X$. 
Acccording to~\cite{LP22-2}, the link $L$ and the decomposition $\H$ are called {\em horizontal} 
if, for some Heegaard decomposition $H_g\cup\Si_g\x [0,1]\cup H'_g$ of $\del_-^{h_1} X$ 
we have: 
\begin{itemize}
\item 
$\ga_i\subset\fS_i:=\Si_g\x\{t_i\}$ for some $0<t_1<\cdots <t_{h_2}<1$; 
\item 
the $i$-th $2$-handle of $\H$ is attached 
to $\ga_i$ with framing $\fr(\ga_i) = \fr_{\fS_i}(\ga_i)\pm 1$.
\end{itemize}
Note that the $3$-manifold 
$\del_+^{h_3} X$ is obtained from $\del_-^{h_1} X$ by $3$-dimensional 
surgery on each $\ga_i$ with the same framing $\fr(\ga_i)$. We say 
that $\H$ is a horizontal decomposition of type $(g,h_1,h_2,h_3)$, or a 
$(g,h_1,h_2,h_3)$-decomposition. As observed in~\cite[Section~1]{LP22-2}, 
the inequality $g\geq\max\{h_1,h_3\}$ always holds. We say that a smooth, 
closed $4$-manifold $\hat X$ has a horizontal decomposition 
if $X:=\hat X\setminus\{B^4\cup B^4\}$ has a such a decomposition 
when viewed as a cobordism $S^3\to S^3$.

Together with~\cite[Theorem~1.7]{LP22-2}, the following result yields a complete 
classification of the smooth, closed, oriented $4$-manifolds $\hat X$ having Euler characteristic 
$\chi(\hat X)\leq 3$ and a horizontal decomposition of genus $1$.

\begin{thm}\label{t:type11} 
Let $\hat X$ be a smooth, closed $4$-manifold admitting a $(1,h_1,h_2,1)$-decomposition. 
Then, $\hat X$ is diffeomorphic, possibly after reversing orientation, to: 
\[
S^4,\quad S^1\x S^3,\quad \pm\CP^2\#S^1\x S^3,\quad \pm\CP^2\#\pm\CP^2\#S^1\x S^3
\quad\text{or}\quad B_{p,q}\cup -B_{p,q} \quad\text{if $\chi(\hat X)\leq 2$,}
\]
\[
\CP^2,\quad \CP^2\# B_{p,q}\cup -B_{p,q},\quad \CP^2\# 2\overline{\CP}^2\# S^1\x S^3, 
\quad\text{or}\quad 3\CP^2\#  S^1\x S^3\quad\text{if $\chi(\hat X)=3$,}
\]
for some $p>q\geq 0$.
\end{thm} 
From now on we concentrate on $(1,h_1,h_2,1)$-decompositions of a $4$-manifold $\hat X$ assuming $\hat X\cong\pm\CP^2$. 
Then, we have $3 = \chi(\hat X) = 1 - h_1 + h_2$. 
Since $h_1\leq g=1$, the pair $(h_1,h_2)$ is equal to either $(0,2)$ or $(1,3)$.
The proof of~\cite[Lemma~5.1]{LP22-2} shows that a $(1,1,3,1)$-decomposition of a $4$-manifold $X$ induces a smooth embedding in $X$ of the disjoint union of $3$ $4$-manifolds, each consisting of a $1$-handle and a $2$-handle attached to $\del_+(S^1\x D^3\setminus B^4) = S^1\x S^2$ along a simple closed curve sitting on the standard genus-$1$ Heegaard torus. 
Using this fact, in the following Theorem~\ref{t:disj-emb} we describe essentially all the smooth embeddings of $B_{p,q}$'s induced by $(1,1,3,1)$-decompositions of $\CP^2$.

We introduce some notation in order to give a more concise statement of Theorem~\ref{t:disj-emb}. 
When we write $B_{p,q}$ we shall always assume that $p$ and $q$ are coprime integers, but not necessarily satisfying $p>q\geq 0$. 
When $p\neq 0$, by $B_{p,q}$ we shall denote the rational homology ball $B_{|p|,\bar q}$, where $\bar q$ is the unique integer congruent to $q$ modulo $|p|$ and satisfying $|p|>\bar q\geq 0$. 
Note that when $p=\pm 1$ then $q=0$ and $B_{\pm 1,0} = B^4$. When $p=0$ then $q=\pm 1$ and by $B_{p,q}$ we shall denote $S^1\x D^3\#q\CP^2$. 

Let $\{x_n(a,b)\}_{n\geq 0}$ be the integer sequence defined by 
\[
x_0(a,b)=a,\quad x_1(a,b)=b,\quad x_{n+2}(a,b) = x_{n+1}(a,b) + x_n(a,b),\quad n\geq 2,
\]
and denote by $F_n := x_n(0,1)$ and $L_n := x_n(2,1)$ respectively the $n$-th 
Fibonacci and Lucas number. 

\begin{thm}\label{t:disj-emb} 
A disjoint union $\pm B_{p_1,q_1}\cup\pm B_{p_2,q_2}\cup\pm B_{p_3,q_3}$ with 
$\max\{|p_i|\}>1$ is induced by a $(1,1,3,1)$-decomposition of 
$\CP^2$ if and only it coincides with one of the following:
\begin{enumerate}
\item[(1)] 
$B_{p_1,q_1}\cup B_{p_2,q_2}\cup B_{p_3,q_3}$, where $(p_1,p_2,p_3)$ is a Markov triple and 
$p_k q_i\equiv \pm 3 p_j\bmod p_i$,
where $(i,j,k)$ is a permutation of $(1,2,3)$; 
\item[(2)]
$B_{x_1,q_1}\cup -B_{x_2,q_2}\cup -B_{x_3,q_3}$, where $x_1^2+x_1 x_2 x_3 = x_2^2+x_3^2$, 
\[
x_2 q_1\equiv\pm x_3 \bmod x_1,\quad  
x_3 q_2\equiv\pm x_1\bmod(x_1 x_3 + x_2),\quad\text{and}\quad 
x_2 q_3\equiv\pm x_1\bmod x_3;
\]
\item[(3)]
$B_{F_a,q_1}\cup -B_{F_{b-\ep a},q_2}\cup B_{F_b,q_3}$
where $a$ and $b$ are odd, coprime integers, $\ep\in\{\pm 1\}$, $b \neq \ep a$, 
\[
q_1\equiv\pm L_b/F_{b-\ep a}\bmod F_a,\quad   
q_2\equiv\pm L_b/F_a\bmod F_{b-\ep a}\quad\text{and}\quad 
q_3\equiv\pm L_a/F_{b-\ep a}\bmod F_b.
\]
\end{enumerate}
\end{thm} 

\begin{rmk}\label{r:knownemb} 
The fact that the disjoint unions of Theorem~\ref{t:disj-emb}(1) embed 
in $\CP^2$ is already known~\cite{HP10}. 
\end{rmk}

Denote by $\F_i$, for $i=1,2,3$, the family of rational homology balls 
$B_{p,q}$ which admit an orientation-preserving embedding in $\CP^2$ in view of  
Theorem~\ref{t:disj-emb}(i). 

\begin{thm}\label{t:manyballs}
Suppose that $|p|>1$ and the rational ball $B_{p,q}$ admits an 
orientation-preserving embedding in $\CP^2$ induced by  
a $(1,h_1,h_2,1)$-decomposition. Then, 
\begin{enumerate} 
\item[(1)]
$B_{p,q}\in\F_1$ if and only if $q^2\equiv -9\bmod p$;
\item[(2)] 
$B_{p,q}\in\F_2$ if and only if $q^2\equiv -1\bmod p$;
\item[(3)] 
let 
\[
S_p = \{x\in\Z/p\Z\ |\ x\equiv q^2\bmod p\ \text{for some $B_{p,q}\in\F_3$}\}.
\]
Then, the sequence $\{|S_p|\}_p$ is unbounded.
\end{enumerate} 
\end{thm}  

\begin{rmk}\label{r: }
In~\cite{LP22-1} we constructed infinitely many embeddings $B_{p,q}\hra\CP^2$, extending a family of embeddings constructed by Owens~\cite{Ow20}. 
It is not hard to check that our constructions implicitly used horizontal decompositions of $\CP^2$ -- in fact, those constructions led us to the discovery of horizontal decompositions. 
At the end of~\cite[Section~3]{LP22-1} we showed that all the balls shown to embed in $\CP^2$ satisfy $q^2\equiv -1\bmod p$. 
In view of Remark~\ref{r:knownemb}, Theorem~\ref{t:manyballs} implies that all $B_{p,q}$'s known so far to embed in $\CP^2$ belong to $\F_1\cup\F_2$ and that $(1,h_1,h_2,1)$-decompositions of $\CP^2$ induce infinitely many more embeddings of $B_{p,q}$'s.  
\end{rmk}

\begin{cor}\label{c:almost-complex}
Let $B_{p,q}$ be a rational ball which admits an orientation-preserving embedding in $\CP^2$ induced by  a $(1,h_1,h_2,1)$-decomposition.
Then, $B_{p,q}$ admits an almost-complex embedding in $\CP^2$ if and only if it embeds symplectically in $\CP^2$ and belongs to $\F_1$.  
\end{cor} 

\begin{proof}
By~\cite[Corollary~1.3]{LP20}, if $B_{p,q}$ admits an almost-complex embedding in $\CP^2$ then $q\equiv -9\bmod p$, therefore by Theorem~\ref{t:manyballs} $B_{p,q}\in\F_1$ and by Remark~\ref{r:knownemb} it embeds symplectically in $\CP^2$. The converse 
follows from the well-known fact that symplectic embeddings are almost-complex.
\end{proof}

The paper is organized as follows. In Section~\ref{s:chileq2} we prove Theorem~\ref{t:type11} when $\chi(\hat X)\leq 2$ and in Section~\ref{s:chi=3} we prove Theorem~\ref{t:type11} when $\chi(\hat X) = 3$. In Section~\ref{s:disj-emb} we prove Theorem~\ref{t:disj-emb} and in Section~\ref{s:manyballs} we prove Theorem~\ref{t:manyballs}.

\begin{ack}
The authors wish to thank the referee for their accurate report, which helped us to improve the exposition. The present work is part of the MIUR-PRIN projects 2017JZ2SW5 and 2022NMPLT8. The first author is a member of the GNSAGA research group (INdAM).
\end{ack}

\section{Proof of Theorem~\ref{t:type11} when $\chi(\hat X)\leq 2$}\label{s:chileq2}

Let $\hat X$ be a smooth, oriented cobordism  
admitting a $(1,h_1,h_2,1)$-decomposition. In this section we prove the first 
part of Theorem~\ref{t:type11}, that is we show that if  
$\chi(\hat X)\leq 2$ then $\hat X$ is diffeomorphic, possibly after 
reversing orientation, to: 
\[
S^4,\quad S^1\x S^3,\quad \pm\CP^2\#S^1\x S^3,\quad \pm\CP^2\#\pm\CP^2\#S^1\x S^3
\quad\text{or}\quad B_{p,q}\cup -B_{p,q} 
\]
for some $p>q\geq 0$.

We start recalling that to an $n$-component link $L\subset S^1\x S^2$, horizontal 
with respect to the standard genus-$1$ Heegaard decomposition of $S^1\x S^2$, one can associate 
an element of mapping class group $\Map(\Si_1)$ of the genus-$1$ surface, factorized as a product 
of Dehn twists~\cite[\S 1]{LP22-2}. This is done as follows.  
Let $T\subset S^1\x S^2$ be the standard Heegaard torus and $\ga_1,\ldots, \ga_n$
the components of $L$, viewed as sitting on $T$. Then, to $\ga_i$ we associate 
the Dehn twist $\tau_i := \tau_{\ga_i}\in \Map(T)$ for $i=1,\ldots, n$. 
By definition, the {\em factorization} associated to $L$ is 
\[
F_L = (\tau_n^{\de_n},\ldots,\tau_1^{\de_1}), 
\]
where $\de_i = \fr_{\fS_i}(\ga_i) - \fr(\ga_i)\in\{\pm 1\}$. 
The product $m_L = \tau_n^{\de_n}\cdots\tau_1^{\de_1}$ is called the {\em monodromy} of $L$. 

Now we prove an auxiliary result, which will be also used later on. 
Recall (see e.g.~\cite[Section~1.2]{Sc22}) that each handlebody decomposition $\H$ 
on a smooth, oriented cobordism $X\co\del_- X\to\del_+ X$ can be viewed as coming 
from a Morse function $f\co X\to [0,1]$. The handlebody decomposition $\overline{\H}$ 
associated with $1-f\co X\to [0,1]$ is sometimes called the {\em dual} of $\H$. 
When $X$ has dimension $4$ each $i$-handle of $\H$ corresponds to a $4-i$-handle of 
$\overline{\H}$. In particular, the co-cores of the $2$-handles of $\H$ are the 
cores of the $2$-handles of $\overline{\H}$, and vice-versa. 

\begin{lemma}\label{l:upside-down}
Let $X\co\del_- X\to \del_+ X$ be a smooth, oriented $4$-dimensional cobordism 
with a horizontal handlebody decomposition $\H$. Let $\overline{\H}$ be the dual 
handlebody decomposition on $X$, viewed as an oriented cobordism 
$\overline{X}\co -\del_+ X\to -\del_- X$. Then, $\overline{\H}$ can be made horizontal 
via a framed isotopy of the attaching circles of its $2$-handles. Moreover, if $\H$ 
is of type $(g,h_1,h_2,h_3)$ then $\overline{\H}$ is of type $(g,h_3,h_2,h_1)$. 
\end{lemma}

\begin{proof}
We assume that we are in the setup of Section~\ref{s:intro}. 
Recall that the $i$-th $2$-handle of $\H$ is attached to $\ga_i$ 
with framing $\fr(\ga_i) = \fr_{\Si_i}(\ga_i)\pm 1$ and as a result the $3$-manifold 
$\del_+^{h_3} X$ is obtained from $\del_-^{h_1} X$ by $3$-dimensional 
surgery on each $\ga_i$ with framing $\fr(\ga_i)$. Consider, for each $i=1,\ldots, h_2$, 
a parallel copy $\ga'_i$ of $\ga_i$ inside $\Si_i$. We may assume that the surgeries 
are performed 
using tubular neighborhoods of the $\ga_i$'s disjoint from the $\ga'_i$'s, and therefore 
that $\ga'_i\subset \del_+^{h_3} X$. It is easy to check that, once endowed with the 
framing $\fr(\ga'_i) := \fr_{\Si_i}(\ga'_i)\mp 1$, the framed knot $\ga'_i$ is framed isotopic
to a $0$-framed meridian of $\ga_i$ -- also viewed inside $\del_+^{h_3} X$. In other 
words, the pair $(\ga'_i,\fr(\ga'_i))$ prescribes, up to framed isotopy, 
how to attach the $i$-th $2$-handle of $\overline{\H}$. Clearly the link 
$L' = \bigcup_{i=1}^{h_2} \ga'_i\subset\del_+^{h_3} X$ is horizontal with respect to 
a Heegaard decomposition of $-\del_+^{h_3} X$ and the statement follows easily.
\end{proof}

We now continue with the first part of Theorem~\ref{t:type11}. 
Let $X$ be the complement of two disjoint balls in $\hat X$, viewed as a cobordism $S^3\to S^3$ and endowed with a horizontal decomposition. 
Clearly $\chi(X) = h_2 - h_1 - 1$.

Assume first $\chi(X)<0$. Since $h_1\leq g=1$, $\chi(X)<0$ implies $(h_1,h_2)\in\{(0,0),(1,0),(1,1)\}$. 
The case $(h_1,h_2)=(0,0)$ is not possible because a single $3$-handle cannot give a cobordism from $S^3$ to $S^3$. 
If $(h_1,h_2) = (1,0)$ then $\hat X$ is the double of $S^1\x D^3$, i.e.~$\hat X\cong S^1\x S^3$.  
If $(h_1,h_2) = (1,1)$ the $2$-handle defines a cobordism $X_\ga\co S^1\x S^2\to S^1\x S^2$, where $\ga\subset S^1\x S^2$ is a simple closed curve sitting on the standard genus-$1$ Heegaard torus $T$ of $S^1\x S^2$ and $\fr_T(\ga)\pm 1$. 
Then, $\ga$ is the boundary of a compressing disc in $H_1$ and it is easy to check that $\hat X\cong \pm\CP^2\# S^1\x S^3$. 

If $\chi(X)=0$ the pair $(h_1,h_2)$ is equal to either $(0,1)$ or $(1,2)$.  
In the first case, by Lemma~\ref{l:upside-down} the upside-down cobordism $\overline{X}\co S^3\to S^3$ admits a horizontal decomposition of type $(1,1,1,0)$, therefore by~\cite[Theorem~1]{LP22-2} $\overline{X}\cong S^3\x [0,1]$ and $\hat X\cong S^4$. 

Now suppose $(h_1,h_2)=(1,2)$. Then, the $2$-handles define a cobordism $X_L\co S^1\x S^2\to S^1\x S^2$, 
where $L\subset S^1\x S^2$ is a $2$-component link, horizontal with respect to the standard genus-$1$ 
Heegaard decomposition of $S^1\x S^2$. The two components $\ga_1, \ga_2$ of $L$ sit on parallel copies 
of the standard Heegaard torus $T\subset S^1\x S^2$ and the factorization of $L$ is given by 
\begin{equation}\label{e:2factor}
F_L = (\tau_2^{\de_2},\tau_1^{\de_1}), 
\end{equation}
where $\de_i = \fr_T(\ga_i) - \fr(\ga_i)\in\{\pm 1\}$ is equal to minus the relative framing of $\ga_i$ and the monodromy of $L$ is $m_L = \tau_2^{\de_2}\tau_1^{\de_1}$ and $\tau_i := \tau_{\ga_i}$, $i=1,2$.

Let $\la\subset T$ be a simple closed curve which is the boundary of a compressing disc in $H_1$. 
Then, $(\la,\la)$ is a Heegaard diagram for $S^1\x S^2$ and by~\cite[Lemma~3.3]{LP22-2} $(m_L(\la),\la)$ is a Heegaard 
diagram for $\del_+ X_L = S^1\x S^2$. In order to have $H_1(\del_+ X_L;\Z)=\Z$ we need 
$m_L(\la)\cdot\la =0$. From now on we fix arbitrary
orientations of $\ga_1$ and $\ga_2$ and we abuse notation, 
denoting with $\ga_1$ and $\ga_2$ also 
the corresponding homology classes in $H_1(T;\Z)$. Define 
\[
x := \ga_1\cdot\la,\quad y := \ga_2\cdot\la\quad\text{and}\quad n := \ga_2\cdot\ga_1.
\]
If $x=y=0$ we have $\ga_1=\ga_2=\pm\la$ and $X$ has a handlebody decomposition 
with a $1$-handle, two $2$-handles attached along $\pm 1$-framed unlinked trivial knots in 
$S^1\x S^2$ and a $3$-handle. It is easy to check that $\hat X$ is orientation-preserving diffeomorphic 
to $-\de_1\CP^2\#-\de_2\CP^2\# S^1\x S^3$. Hence, from 
now on we assume $(x,y)\neq (0,0)$. Moreover, we choose the orientations of $\ga_1$ and $\ga_2$
so that $n\geq 0$. Since 
\[
m_L(\la) = \tau_2^{\de_2}(\tau_1^{\de_1}(\la)) = 
\tau_2^{\de_2}(\la + \de_1 x \ga_1) = 
\la + \de_1 x \ga_1 + \de_2 y \ga_2 + \de_1\de_2 n x \ga_2, 
\]
taking the intersection product with $\de_1\de_2\la$ yields  
\begin{equation}\label{e:quadratic} 
\de_2 x^2 + \de_1 y^2 + nxy = 0.
\end{equation}
Equation~\eqref{e:quadratic} implies that $x$ and $y$ divide each other, 
therefore $|x|=|y|\neq 0$, and dividing by $|xy| = x^2 = y^2$ it follows that 
\[
\de_2 + \de_1 \pm n = 0.
\]
We conclude $n\in\{0,2\}$. If $n=0$ we have $\ga_1=\pm\ga_2$ and $\de_1\de_2=-1$, 
therefore $F_L$ is either of the form $(\tau_\ga,\tau_\ga^{-1})$ or $(\tau_\ga^{-1},\tau_\ga)$, 
and in both cases $m_L$ is the identity. In fact, the components $\ga_1$ and $\ga_2$ of $L$, 
which are the attaching circles of the two $2$-handles, cobound an annulus $A$ of 
the form $\ga\x [0,1]\subset T\x [0,1]\subset Y_1$, 
and they have framings with respect to $A$ which are opposite and equal to $1$ in absolute value. 
Sliding for instance $\ga_1$ over $\ga_2$ changes $L$ into the framed link consisting of $\ga_2$ and 
a $0$-framed meridian of $\ga_2$. By~\cite[Example~4.6.3]{GS}, this shows that $\hat X$ is the double of 
$\hat X_{\ga_2} := S^1\x D^3\cup X_{L_2}$. We can view $\hat X_{\ga_2}\setminus B^4$ as a cobordism $S^3\to\del_+ X_{\ga_2}$
with a horizontal decomposition of type $(1,1,1,0)$. By~\cite[Proposition~3.1]{LP22-2} it follows that 
$\hat X_{\ga_2}\cong\pm B_{p,q}$ for some $p>q\geq 0$. Therefore $\hat X\cong B_{p,q}\cup -B_{p,q}$.

If $n=2$ we must have $\de_1=\de_2=-\sgn(xy)$, so that $\ga_1=x \mu + a \la$ and  
$\ga_2=-\de_1 (x \mu+b \la)$ for some $a,b \in \mathbb{Z}$ (both coprime with $x$), 
where $\mu,\la\subset T$ is a symplectic basis and 
$x$ can be assumed to be positive up to reversing the orientation of both $\ga_1$ and $\ga_2$. 
Now, $n=\ga_2 \cdot \ga_1=2$ implies $-\de_1 (a-b)x=2$. We observe that $x=2$ gives a contradiction, 
because this would force both $a$ and $b$ to be odd, and the equality $-\de_1 (a -b)=1$ would be 
false modulo $2$. Therefore $x=1$ and $b-a=2\de_1$. Twisting the 1-handle as in~\cite[Figure~3]{LP22-2} 
replaces $\mu$ with $\mu$ plus a multiple of $\la$. This allows us to assume $\ga_1=\mu$ and, 
consequently, $\ga_2=-\de_1(\mu+(b-a)\la)=-\de_1 \mu-2\la$. A Kirby calculus picture of $\hat X$ 
consists of a dotted unknot $U$, a $-\de_1$-framed meridian of $U$ corresponding to 
$\tau_1^{\de_1}$ and another meridian of $U$, corresponding to $\tau_2^{\de_2}$, whose 
Seifert framing is $2\de_1 - \de_1 = \de_1$. Sliding one meridian over the other 
we get a picture for the double of $B^4$, therefore in this case $\hat X = S^4$. 

\section{Proof of Theorem~\ref{t:type11} when $\chi(\hat X)=3$}\label{s:chi=3}

The purpose of this section is to prove the following.

\begin{thm}\label{t:chi=3}
Let $\hat X$ be a smooth, oriented cobordism admitting a $(1,h_1,h_2,1)$-decomposition and such that $\chi(\hat X)=3$. 
Then, $\hat X$ is diffeomorphic, possibly after reversing orientation, to 
\[
\CP^2,\quad \CP^2\# B_{p,q}\cup -B_{p,q},\quad \CP^2\# 2\overline{\CP}^2\# S^1\x S^3, 
\quad\text{or}\quad 3\CP^2\#  S^1\x S^3
\]
for some $p>q\geq 0$.
\end{thm} 

Let $X$ be the complement of two disjoint balls in $\hat X$, viewed as a cobordism $S^3\to S^3$ and endowed with the horizontal decomposition.
Then, we have $1 = \chi(X) = -h_1 + h_2 - 1$. 
Since $h_1\leq g=1$, the pair $(h_1,h_2)$ is equal to either $(0,2)$ or $(1,3)$. In the first case, the $2$-handles define a cobordism $X_L\co S^3\to S^1\x S^2$, where $L\subset S^3$ is a $2$-component link, horizontal with respect to the standard genus-$1$ Heegaard decomposition of $S^3$. 
By Proposition~\ref{l:upside-down} the upside-down cobordism $\overline{X}\co S^3\to S^3$ admits a horizontal decomposition of type $(1,1,2,0)$, therefore by~\cite[Theorem~1]{LP22-2} $\hat X\cong\pm\CP^2$. 
This proves Theorem~\ref{t:chi=3} when $(h_1,h_2)=(0,2)$. 

When $(h_1,h_2)=(1,3)$ the proof of Theorem~\ref{t:chi=3} is much more involved and is organized as follows. 
In Section~\ref{ss:trace} we prove an auxiliary result and in Section~\ref{ss:equations} we use it to show that to the monodromy factorization of a horizontal decomposition are associated solutions of certain Diophantine equations -- Equations~\eqref{e:eq1} and~\eqref{e:eq2}. Recall that, given an $n$-tuple $F=(g_n,\ldots,g_1)$ of elements of a group $G$, a {\em Hurwitz move} on $F$ is a transformation of the form 
\[
F = (\ldots, g_{i+1},g_i,\ldots)\mapsto F'=(\ldots, g_{i+1} g_i g_{i+1}^{-1}, g_{i+1}, \ldots)
\]
or 
\[
F = (\ldots, g_{i+1},g_i,\ldots)\mapsto F'=(\ldots, g_i, g_i^{-1} g_{i+1} g_i, \ldots).
\]
In Section~\ref{ss:hurwitz} we show how the solutions associated to Equation~\eqref{e:eq1} change under the application of a Hurwitz move to the factorization. Hurwitz moves do not change the cobordism $X$~\cite[Prop.~4.6]{LP22-2}, so the basic idea is to use the equations to find a sequence of Hurwitz moves which simplify the factorization. 
We also show how sets of solutions corresponding to various parameters are related, we define the notion of weakly minimal solution and we observe that applying Hurwitz moves each solution can be made weakly minimal. As for~\cite[Theorem~1.7]{LP22-2}, the inspiration for this proof came from a paper by Denis Auroux~\cite{Au15}. In Section~\ref{ss:111} we determine, for the relevant values of the parameters, the weakly minimal solutions of Equation~\eqref{e:eq1}. 
Moreover, for each such solution we recover the corresponding ``weakly minimal'' factorization. 
In Sections~\ref{ss:111-a=0} and~\ref{ss:111-a=4} we determine the cobordisms represented by the weakly minimal factorizations and identify each closed $4$-manifold $\hat X$ from the corresponding Kirby diagram. 
We proceed in a similar way in Sections~\ref{ss:-111}, \ref{ss:-111-a=0} and~\ref{ss:-111-a=4}. 
Equation~\eqref{e:eq2} is not used in this section but will be used in Section~\ref{s:disj-emb} to prove Theorem~\ref{t:disj-emb}.

\subsection{A trace formula}\label{ss:trace} 

Let $(V,\cdot)$ be a 2-dimensional, real symplectic vector space. Given $v\in V$, denote 
by $\tau_v: V\to V$ the transvection defined by 
$$
\tau_v(w) = w + (v\cdot w) v,\quad w\in V.
$$

Note that $\tau^n_v(w) = w + n(v\cdot w) v$ for each $n\in\Z$. 

\begin{lemma}\label{l:traceformula}
Given $u,v,w\in V$ and $l,n, m\in\Z$, let $\Phi = \tau^l_u\tau^n_v\tau^m_w$. Then,    
\[
nm(v\cdot w)^2+lm(u\cdot w)^2+ln(u\cdot v)^2 = lmn(w\cdot u)(w\cdot v)(v\cdot u) + 2 - \tr(\Phi).
\]
\end{lemma}

\begin{proof}
Once we fix $l,n,m \in \mathbb{Z}$, the two sides of the equality are continuous functions 
of $u$, $v$ and $w$, therefore it is enough to prove that they agree when $v$ and $w$ 
are linearly independent. Using $\{v,w\}$ as a basis of $V$, we easily see that for 
any endomorphism $f$
\[
\tr(f)=\dfrac{f(v) \cdot w+v \cdot f(w)}{v \cdot w}
\]

If we write $\tau_v^n \tau_w^m(v)=av+bw$ and $\tau_v^n \tau_w^m(w)=cv+dw$, we have
\[
\begin{split}
\tau_u^l \tau_v^n \tau_w^m(v)=\tau_u^l (av+bw)=a(v+l(u \cdot v)u)+b(w+l(u \cdot w)u)\\
\tau_u^l \tau_v^n \tau_w^m(w)=\tau_u^l (cv+dw)=c(v+l(u \cdot v)u)+d(w+l(u \cdot w)u)
\end{split}
\]
so that
\[
\dfrac{\tau_u^l \tau_v^n \tau_w^m(v) \cdot w + v \cdot \tau_u^l \tau_v^n \tau_w^m(w)}{v \cdot w}=a+d+l\left(\dfrac{b}{v \cdot w}(u \cdot w)^2+\dfrac{a-d}{v \cdot w}(u \cdot v)(u \cdot w)-\dfrac{c}{v \cdot w}(u \cdot v)^2\right)
\]
Finally, we compute $a=1-mn(v \cdot w)^2$, $b=-m(v \cdot w)$, $c=n(v \cdot w)$ and $d=1$. 
Plugging in theses values we get the desired equality.
\end{proof}

\subsection{Diophantine equations}\label{ss:equations} 
As explained in~\cite[\S 1]{LP22-2}, a horizontal link 
$L = \bigcup_{i=1}^3 \ga_i\subset S^1\x S^2$ determines a factorization of the form
\[
F_L=(\tau_3^{\de_3},\tau_2^{\de_2},\tau_1^{\de_1}), 
\]
where $\tau_i$, $i=1,2,3$, is a positive Dehn twist along $\ga_i$ and $\de_i = \fr_T(\ga_i) - \fr(\ga_i)\in\{\pm 1\}$. 
The 2-handles produce a cobordism from $S^1 \times S^2$ to $S^1 \times S^2$ and, 
as observed in Section~\ref{s:chileq2},  
this happens if and only if $m_L(\la)\cdot \la =0$. 
Let $\mu,\la\subset T$ be a symplectic basis of oriented, simple closed curves in the standard Heegaard torus $T\subset S^1\x S^2$. 
From now on we fix arbitrary orientations of the $\ga_i$'s and we abuse notation 
by denoting with $\ga_i$ also the corresponding homology classes in $H_1(T;\Z)$.
We have $\ga_i = p_i\mu + q_i\la$, $i=1,2$, for some coprime $p_i, q_i\in\Z$. 
Let us introduce the variables $x_1,x_2,x_3$ by setting 
\begin{equation}\label{e:relations1}
(x_1,x_2,x_3) = (\ga_2\cdot\ga_3, \ga_1\cdot\ga_3,\ga_1\cdot\ga_2). 
\end{equation} 
Then, the trace formula of Lemma~\ref{l:traceformula} reads 
\begin{equation}\label{e:eq1}
\de_1 x_1^2 + \de_2 x_2^2 + \de_3 x_3^2 -  x_1 x_2 x_3 = \de_1\de_2\de_3 (2 - \tr(m_L)).
\end{equation} 
Moreover, observe that $m_L(\la)\cdot \la =0$ implies $m_L(\la) = \pm \la$, therefore 
$\det(m_L) = 1$ implies $\tr(m_L) = \pm 2$. Also, we have 
\begin{multline*}
m_L(\la)\cdot \la = \tau_3^{\de_3}\tau_2^{\de_2}\tau_1^{\de_1}(\la) = 
\tau_3^{\de_3}\tau_2^{\de_2}(\la + \de_1 p_1\ga_1) = 
\tau_3^{\de_3}(\la + \de_1 p_1\ga_1 + \de_2(p_2 - \de_1 p_1 x_3)\ga_2) \\ 
= \la + \de_1 p_1\ga_1 + (\de_2 p_2 - \de_1\de_2 p_1 x_3)\ga_2 + 
\de_3(p_3 - \de_1 p_1 x_2 - \de_2 p_2 x_1 + \de_1\de_2 p_1 x_1 x_3)\ga_3.
\end{multline*}
Therefore, the condition $m_L(\la)\cdot \la =0$ translates into 
\begin{equation}\label{e:eq2}
\de_1 p^2_1 + \de_2 p^2_2 + \de_3 p^2_3
- \de_1\de_2 p_1 p_2 x_3 - \de_1 \de_3 p_1 p_3 x_2 - \de_2 \de_3 p_2 p_3 x_1 + \de_1\de_2\de_3 p_1 p_3 x_1 x_3 =0.
\end{equation} 
In the following sections we study the following sets of solutions of Equation~\eqref{e:eq1}:
\[
S_a^{\de_1,\de_2,\de_3}:=
\left\{
(\de_1,\de_2,\de_3,x_1,x_2,x_3) \in \{ \pm 1\}^3 \times \mathbb{Z}^3 
\ | \ \de_1 x_1^2 + \de_2 x_2^2 + \de_2 x_3^2 -  x_1 x_2 x_3 = \de_1\de_2\de_3 a
\right\}
\]
where $a:=2-\tr(m_L)$ is either $0$ or $4$. 
We shall provide a proof of Theorem~\ref{t:chi=3} in the case $(h_1,h_2)=(1,3)$ by combining 
the information on the sets $S_a^{\de_1,\de_2,\de_3}$ with the constraint provided by Equation~\eqref{e:eq1}. 
Equation~\eqref{e:eq2} will be used in Section~\ref{s:disj-emb}.

\subsection{Factorizations and Hurwitz moves}\label{ss:hurwitz} 

In this section we show that to understand the sets of solutions $S_a^{\de_1,\de_2,\de_3}$ for all values of the $\de_i$'s it suffices to study the sets $S_a^{1,1,1}$ and $S_a^{-1,1,1}$ with $a \in \{0,4\}$. 
The latter sets will be analyzed in Sections~\ref{ss:111} and~\ref{ss:-111}. 

Let us start with the observation that changing the orientation of one of the curves $\ga_i$ changes the signs of two of the elements in the associated triple $(x_1,x_2,x_3)$, providing another solution of Equation~\eqref{e:eq1}. 
We define the {\em sign changes} as follows:
\[
(x_1,x_2,x_3)\mapsto (-x_1,-x_2,x_3),\quad (x_1,x_2,x_3)\mapsto (-x_1,x_2,-x_3),
\quad (x_1,x_2,x_3)\mapsto (x_1,-x_2,-x_3).
\label{e:sign-changes}
\]
Next, we observe that applying a Hurwitz move we get 
\[
m_L = 
\tau_{\ga_3}^{\de_3}\tau_{\ga_2}^{\de_2}\tau_{\ga_1}^{\de_1} = 
\tau_{\tau_3^{\de_3}(\ga_2)}^{\de_2}\tau_{\ga_3}^{\de_3}\tau_{\ga_1}^{\de_1} = 
\tau_{-\tau_3^{\de_3}(\ga_2)}^{\de_2}\tau_{\ga_3}^{\de_3}\tau_{\ga_1}^{\de_1}.
\]
To see the effect on a solution $(x_1,x_2,x_3)$ of Equation~\eqref{e:eq1} note that 
we are replacing $(\ga_3,\ga_2,\ga_1)$ with $(-\tau_{\ga_3}^{\de_3}(\ga_2),\ga_3,\ga_1)$
and since $\tau_3^{\de_3}(\ga_2) = \ga_2+\de_3(\ga_3\cdot\ga_2)\ga_3$
the associated triple becomes $(x_1,\de_3 x_1 x_2-x_3, x_2)$.
Setting $\hat x_3 := \de_3 x_1 x_2 - x_3$, the triple  
$(\de_1,\de_3,\de_2,x_1,\hat x_3,x_2)$ satisfies Equation~\eqref{e:eq1}. Similarly, the other Hurwitz 
moves lead to the solutions 
\[
(\de_1,\de_3,\de_2,x_1, x_3, \hat x_2),\quad 
(\de_2,\de_1,\de_3,x_2,\hat x_1,x_3)\quad\text{and}\quad 
(\de_2,\de_1,\de_3,\hat x_2,x_1,x_3), 
\]
where 
$\hat x_1 := \de_1 x_2 x_3 - x_1$ and $\hat x_2 := \de_2 x_1 x_3 - x_2$, which also 
satisfy Equation~\eqref{e:eq1}. Call {\em mutations} the following changes 
\[
\begin{split}
\mu_1\co (\de_1,\de_2,\de_3,x_1,x_2,x_3)\mapsto ( \de_1,\de_2,\de_3,\hat x_1, x_2, x_3),\\ 
\mu_2\co (\de_1,\de_2,\de_3,x_1,x_2,x_3)\mapsto (\de_1,\de_2,\de_3,x_1, \hat x_2, x_3),\\ 
\mu_3\co (\de_1,\de_2,\de_3,x_1,x_2,x_3)\mapsto (\de_1,\de_2,\de_3,x_1, x_2, \hat x_3)
\end{split}
\]
and, for any permutation $\si\in S_3$, set 
\[
\si(\bde,\bx) := (\de_{\si(1)},\de_{\si(2)},\de_{\si(3)},x_{\si(1)},x_{\si(2)},x_{\si(3)}).
\]
Then, we see that the Hurwitz moves induce the transformations
\[
(23)\circ\mu_3,\quad 
(23)\circ\mu_2,\quad 
(12)\circ\mu_1\quad\text{and}\quad
(12)\circ\mu_2. 
\]
In particular, each of the three mutations is realized, up to permutations, by at least one Hurwitz move.

\begin{defn}
	A solution $(\de_1,\de_2,\de_3,x_1,x_2,x_3)$ of Equation~\eqref{e:eq1} is {\em minimal} ({\em weakly minimal}, respectively) 
	if $|x_i|< |\hat x_i|$ ($|x_i|\leq |\hat x_i|$ respectively) for each $i=1,2,3$. A factorization $(\tau_{\ga_3}^{\de_3},\tau_{\ga_2}^{\de_2},\tau_{\ga_1}^{\de_1})$ is {\em minimal} or {\em weakly minimal} if so is the corresponding solution of Equation~\eqref{e:eq1} 
\end{defn}

Observe that any solution of Equation~\eqref{e:eq1} can be modified into a weakly minimal solution by a (possibly empty) sequence of mutations, each of which makes the quantity $|x_1|+|x_2|+|x_3|$ decrease strictly.

Now suppose that $(\bde,\bx) = (\de_1,\de_2,\de_3,x_1,x_2,x_3)$ is a solution of Equation~\eqref{e:eq1} and let 
\[
-(\bde,\bx)=(-\de_1,-\de_2,-\de_3,-x_1,-x_2,-x_3).
\]
If we represent $S^1\x S^2$ as a dotted unknot $U \subset S^3$, the union of $U$ with $L$ gives a Kirby 
diagram for $\hat X := B^4\cup_{\del_- X} X$. As already remarked in~\cite[Section~4.4]{LP22-2}, taking the mirror 
image of such a diagram with respect to a plane intersecting $U$ transversely in two points amounts to replacing the curves $\ga_i=p_i \mu + q_i \la$ with 
$\overline{\ga_i}=p_i \mu-q_i\la$ and changing the signs of their relative framings. 
The result is a diagram of a horizontal decomposition of $-X$. 
Observe that:
\begin{itemize}
\item 
If $(\bde,\bx)$ was associated to the $\ga_i$'s, to the $\overline{\ga_i}$'s is associated 
$-(\bde,\bx)$. 
\end{itemize}
Moreover, using the notation of Section~\ref{ss:hurwitz} it is easy to check 
that:
\begin{itemize} 
\item 
since 
\[
(23)\circ\mu_2\circ (12)\circ\mu_1= (123)\quad\text{and}\quad  
(12)\circ\mu_2\circ (23)\circ\mu_3 = (132), 
\]
for each cyclic permutation $\sigma$ of $(1,2,3)$, 
the transformation $(\bde,\bx)\mapsto\si(\bde,\bx)$ is realized by a composition of 
two Hurwitz moves on the factorization.
\end{itemize}
Summarizing, the observations above show that there are bijections 
\[
S_a^{\de_1,\de_2,\de_3} \longleftrightarrow S_a^{-\de_1,-\de_2,-\de_3}\qquad 
S_a^{\de_1,\de_2,\de_3} \longleftrightarrow S_a^{\de_{\sigma(1)},\de_{\sigma(2)},\de_{\sigma(3)}}
\]
with $\si$ any cyclic permutation. 
Moreover, the above bijections are induced by orientation-reversing diffeomorphisms 
in the first case and orientation-preserving diffeomorphisms in the second case. 
Therefore, in order to prove Theorem~\ref{t:chi=3} it will suffice to determine 
the oriented cobordisms underlying the minimal 
factorizations in the cases $(\de_1,\de_2,\de_3)=(1,1,1)$ and $(\de_1,\de_2,\de_3)=(-1,1,1)$. 
Since there are two possibilities for $a=2-\tr(m_L)$, we have four cases in total, 
which we treat separately below.

\subsection{Minimal solutions when $(\de_1,\de_2,\de_3)=(1,1,1)$}\label{ss:111}

In this section we determine the weakly minimal solution of the equation 
\begin{equation}\label{e:cohn} 
x^2+y^2+z^2-xyz=a, 
\end{equation}
for $a$ equal to either $0$ or $4$.

\begin{prop}\label{p:cohn} 
	Let $(x,y,z)$ be a solution of Equation~\eqref{e:cohn} with $|x|$, $|y|$, $|z|$ pairwise distinct. 
	Then, $(x,y,z)$ is not weakly minimal. Indeed, mutation at the greatest element in absolute
	value makes the quantity $|x|+|y|+|z|$ strictly smaller. Moreover, if $\min(|x|,|y|,|z|)>1$ then mutation at each of the other two elements makes the quantity $|x|+|y|+|z|$ strictly bigger. 
\end{prop}

\begin{proof} 
	Up to permuting the variables we may assume:
	$$
	|x| > |y| > |z|.
	$$
	Note that, by~\eqref{e:cohn}, this implies $|z|>0$ (otherwise $x^2+y^2=a \in \{0,4 \}$, which would imply $y=0=z$), therefore we have $xyz\neq 0$.   
	We claim that $|\hat x|<|x|$. After possibly applying sign changes we may assume that 
	$x,y>0$, and since $xyz = x^2+y^2+z^2-a >0$, we have $z>0$ as well. Also, $x\hat{x} = y^2+z^2-a>0$, 
	therefore $\hat{x}>0$. Replacing $x$ with $y$ in the equation gives 
	\[
	2y^2 + z^2 - y^2 z-a < y^2(3-z),
	\]
	therefore if $z\geq 3$ then $y$ is between $x$ and $\hat{x}$, and since $y<x$ we must have $\hat{x}<x$, 
	as claimed. On the other hand, $z=2$ is impossible because it implies  
	\[
	(x-y)^2=a-4,
	\]
    which has no solutions if $a=0$ and forces $x=y$ if $a=4$. Finally, $z=1$ gives:  
    \[
    x^2-xy+y^2=a-1 \le 3
    \]
    but since $x>y>z>0$ we have $x >2$, so that
    \[
    x^2-xy+y^2=\dfrac{3}{4} x^2+ \left( \dfrac{x}{2} -y \right)^2 \ge \dfrac{3}{4} x^2>3
    \]
	This proves the first part of the statement. To prove the second part it suffices to show that if $|x|>|y|>|z|\geq 2$ then $|y|<|\hat{y}|$ and $|z|<|\hat{z}|$. We have 
	$$
	2|y| < |xy| \leq |xy -z| + |z| < |\hat{z}| + |y|,
	$$
	therefore $|z|<|y|<|\hat z|$. Similarly, 
	$$
	2|x|\leq |xz| \leq |xz - y| + |y| < |\hat{y}| + |x|,
	$$
	therefore $|y|<|x|<|\hat y|$. 
\end{proof}

\begin{lemma}\label{l:cohn}
	A weakly minimal solution $(x,y,z)$ of Equation~\eqref{e:cohn} is equal, up to sign changes and permutations, to:
	\begin{itemize}
	    \item $(0,0,0)$ or $(3,3,3)$ if $a=0$;
	    \item $(2,0,0)$, $(-1,1,1)$ or $(2,t,t)$ for some $t \ge 2$ if $a=4$.
	\end{itemize}
\end{lemma} 

\begin{proof}
By Proposition~\ref{p:cohn}, $|x|,|y|,|z|$ are not pairwise distinct. Therefore we can assume, possibly after permuting the variables and changing two signs, that $y=z \ge 0$. Then, if $a=0$ we have 
\[
x^2 = (x-2)y^2.
\]
If $y=0$ we get the solution $(0,0,0)$, otherwise we must have $x = cy$, with $c^2 =x - 2$ for some 
$c\in\Z$. Therefore we have $c^2-cy+2=0$, which implies $c\ |\ 2$. Moreover, $y>0$ implies $c>0$, 
so that $c \in \{1,2\}$. These two values of $c$ lead to the solutions $(3,3,3)$ and $(6,3,3)$, 
and the conclusion when $a=0$ follows easily. If $a=4$ we have 
\[
x^2-4=(x-2)y^2,
\]
which implies either $x=2$ or $x=y^2-2$. It is easy to check that the solution $(2,y,y)$ can only mutate to itself and to $(y^2-2,y,y)$, therefore it is weakly minimal if and only if $|y^2-2| \ge 2$, which holds for all $y \neq 1$; on the other hand, a necessary condition for $(y^2-2,y,y)$ to be weakly minimal is $|y^2-2|\le 2$, i.e. $y \in \{0,1,2\}$. For $y \in \{0,2\}$, up to sign changes we get the 
weakly minimal solutions $(2,0,0)$ and $(2,2,2)$, while for $y=1$ we get $(-1,1,1)$.
\end{proof}

\subsection{Underlying $4$-manifolds when $(\de_1,\de_2,\de_3)=(1,1,1)$ and $a=0$}\label{ss:111-a=0}

The minimal solution $(0,0,0)$ is realized only by a triple $(\ga_1,\ga_2,\ga_3)$ with $\ga_1=\ga_2=\ga_3$, 
and since $m_L = \tau_\ga^3$ must fix $\la$, we have $\ga_i=\la$ for each $i$ and the factorization 
$\tau_{\la}\tau_{\la}\tau_{\la} = \tau_{\la}^3$. It is easy to check from the associated Kirby diagram that 
$\hat X \cong 3 \overline{\CP}^2 \# S^1\x S^3$. 

The following computation shows that the minimal solution $(3,3,3)$ of 
Lemma~\ref{l:cohn} is realized by a triple $(\ga_1,\ga_2,\ga_3)$ with 
$\ga_3 = \mu+2\la$, $\ga_2=2\mu+\la$ and $\ga_1=\mu-\la$, corresponding to 
the factorization $\tau_{\ga_3}\tau_{\ga_2}\tau_{\ga_1} = \tau_\la^{-9}$. Since $\ga_3$ 
is nonseparating there is another curve $\widetilde\ga\subset T^2$ with $\ga_3 \cdot\widetilde\ga = 1$. 
The homology classes of $\ga_3$ and $\widetilde\ga$ generate $H_1(T^2;\Z)$, therefore we 
can express $\ga_1$ and $\ga_2$ in these coordinates. From 
$\ga_1 \cdot \ga_3=\ga_2 \cdot \ga_3=3$ we get 
$\ga_1=a \ga_3-3\widetilde\ga$ and $\ga_2=b \ga_3-3\widetilde\ga$ for some integers $a$ and $b$, 
both coprime with $3$. Moreover, $\ga_1 \cdot \ga_2=3$ implies $b=a+1$, so that 
the remainders of $a$ and $b$ modulo $3$ are forced to be $1$ and $2$ respectively. 
We can then write
\[
\ga_1=(3c+1)\ga_3-3 \de=\ga_3-3\bar\ga\quad\text{and}\quad
\ga_2=(3c+2)\ga_3-3 \de=2\ga_3-3\bar\ga 
\]
for some $c\in\Z$, where $\bar\ga:=\widetilde\ga-c\ga_3$ is another curve satisfying $\ga_3\cdot\bar\ga=1$. Now it is straightforward to check that $m_L(\bar\ga)=\bar\ga$ and $m_L(\ga_3)=\ga_3+9\bar\ga$, so that $\bar\ga=\pm \la$. 
Without loss of generality, up to reversing $\bar\ga$ and the 
$\ga_i$'s we may assume $\bar\ga=\la$. Then, $m_L=\tau_\la^{-9}$. In particular, since $\ga_3 \cdot \la=1$ we can write 
$\ga_3=\mu+k \la$ for some $k\in\Z$. Any choice of $k$ determines $\ga_1$,$\ga_2$ and $\ga_3$, hence 
a Kirby diagram for $\hat X$. Twisting the $1-$handle as in 
Section~\ref{s:chileq2} transforms each $\ga_i$ 
through a power of $\tau_\la$ without changing the underlying $4$-manifold, therefore it suffices to identify $\hat X$ for a single value of $k$. Choosing $k=2$ we get 
\[
\ga_1 = \mu-\la,\quad\ga_2 = 2\mu+\la\quad\text{and}\quad\ga_3 = \mu +2\la.
\]
The resulting Kirby diagram is shown in the first picture from the left in Figure~\ref{f:fig1}, where the presence of a $3$- and a $4$-handle is understood. 
The framings $\fr(\ga_i)$ are obtained as follows. If $\ga_i = p_i \mu + q_i\la$, 
the Seifert framing induced by the Heegaard torus $T$ is $\fr_T(\ga_i) = p_i q_i$. Indeed, 
since this framing is the linking number of $\ga_i$ with its push-off in the 
direction of the negative normal to $T$, we have 
\[
\fr_T(\ga_i) = \lk(\ga_i,p_i\mu^- + q_i\la^-) = \lk(\ga_i, q_i\la^-) = p_i q_i.
\]
Then, since $(\de_1,\de_2,\de_3)=(1,1,1)$, we obtain 
$\fr(\ga_i) = \fr_T(\ga_i) - \de_i = p_i q_i -1$ for each $i=1,2,3$.
\begin{figure}[ht]
\labellist
\hair 2pt
\pinlabel $\ga_1$ at 145 20
\pinlabel $\ga_3$ at 145 90
\pinlabel $\ga_2$ at 18 10
\pinlabel $-2$ at 42 50
\pinlabel $1$ at 74 40
\pinlabel $1$ at 79 75
\pinlabel $-2$ at 231 57
\pinlabel $0$ at 263 53
\pinlabel $1$ at 270 78
\pinlabel $1$ at 407 28
\pinlabel $0$ at 453 45
\pinlabel $0$ at 403 75
\endlabellist
\centering
\includegraphics[width=12cm]{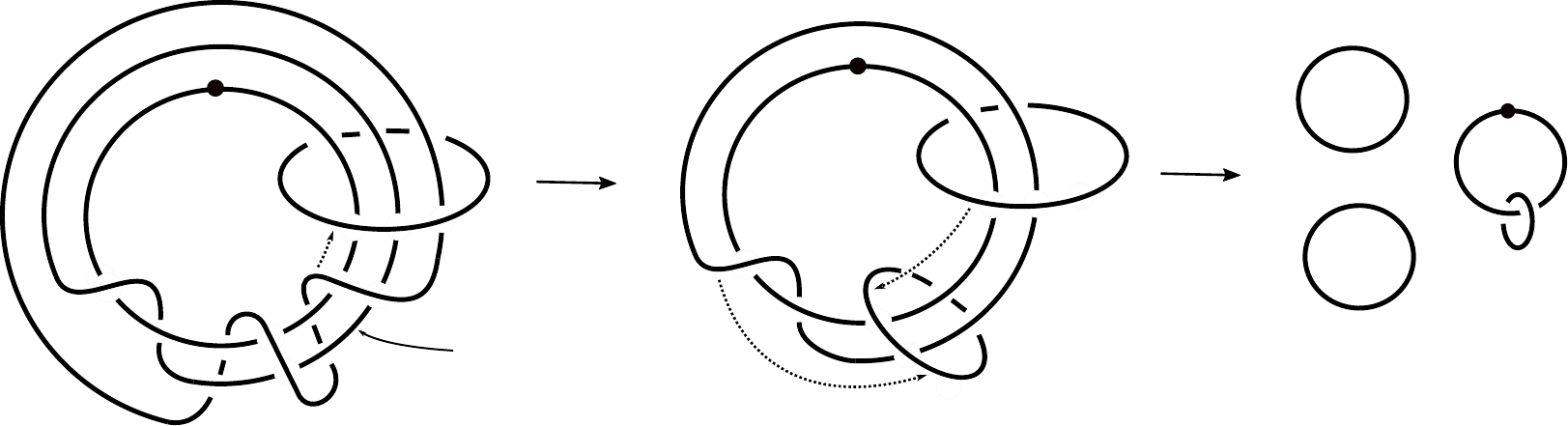}
\caption{Proof that $\hat X\cong\CP^2$ when $(\de_1,\de_2,\de_3)=(1,1,1)$, $a=0$}
\label{f:fig1}
\end{figure}
After the handle slide suggested by the dotted arrow one obtains the central picture of Figure~\ref{f:fig1}. 
Two further handle slides indicated the dotted arrows lead to the picture on the right, where the 
$0$-framed $2$-handle cancels with the $3$-handle. This shows that $\hat X\cong\CP^2$. 

\subsection{Underlying $4$-manifolds when $(\de_1,\de_2,\de_3)=(1,1,1)$ and $a=4$}\label{ss:111-a=4} 

If the solution $(2,0,0)$ was realized by curves $\ga_1$, $\ga_2$ and $\ga_3$ 
then we would have $\ga_1\cdot\ga_2=\ga_1\cdot\ga_3=0$ and therefore 
$\ga_1=\ga_2=\ga_3$ up to signs, which is incompatible with $\ga_2\cdot\ga_3=2$. 
Thus, $(2,0,0)$ cannot be realized.

Now consider the solution $(-1,1,1)$. Since $\ga_1\cdot\ga_2=1$, we may 
take $\ga_1$ and $\ga_2$ as a symplectic basis and write $\ga_3=a\ga_1+b\ga_2$.
Then, $\ga_1\cdot\ga_3=1$ and $\ga_2\cdot\ga_3=-1$ imply $\ga_3=\ga_1+\ga_2$, 
and the monodromy is $m_L = \tau_{\ga_1+\ga_2}\tau_{\ga_2}\tau_{\ga_1}$. But it is  
easy to check that $m_L(\ga_1)=-\ga_1-3\ga_2$ and $m_L(\ga_2)=-\ga_2$. Therefore 
$\ga_2=\pm\la$ and $\ga_1 = \pm\mu + k\la$ for some $k\in\Z$. Up to applying a $1$-handle twist 
which does not change the underlying $4$-manifold as in Section~\ref{ss:111-a=0}, 
we may assume $\ga_1=\pm\mu$, $\ga_2 =\pm\la$ and $\ga_3 = \pm(\mu+\la)$. Then, 
$m_L = -\tau_\la^{-3}$ and the resulting Kirby diagram is given by the first picture from the left in Figure~\ref{f:fig2}, where the presence of a $3$- and a $4$-handle is understood.   
\begin{figure}[ht]
	\labellist
	\hair 2pt
	\pinlabel $-1$ at 10 10
	\pinlabel $-1$ at 55 37
	\pinlabel $\ga_1$ at 85 20
	\pinlabel $\ga_2$ at 110 20
	\pinlabel $\ga_3$ at 127 67
	\pinlabel $0$ at 72 55
	\pinlabel $-1$ at 180 10
	\pinlabel $0$ at 204 70
	\pinlabel $-1$ at 225 25
	\pinlabel $0$ at 308 55
	\pinlabel $-1$ at 280 33
	\pinlabel $-1$ at 288 15
	\pinlabel $\cong$ at 330 50
	\pinlabel $-1$ at 341 68
	\pinlabel $0$ at 345 40
	\endlabellist
	\centering
	\includegraphics[width=12cm]{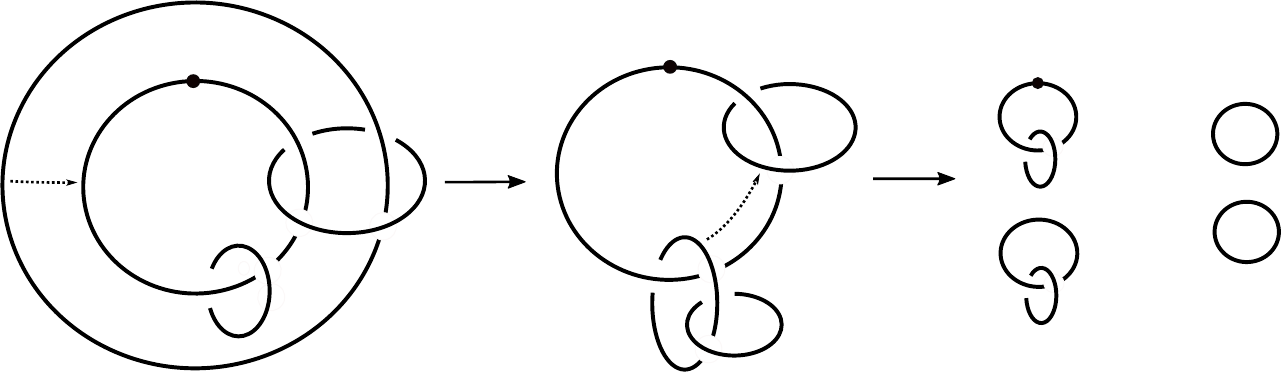}
	\caption{Proof that $\hat X\cong\CP^2$ when $(\de_1,\de_2,\de_3)=(1,1,1)$ and $a=4$}
	\label{f:fig2}
\end{figure}
The indicated handle slides show that $\hat X\cong\overline{\CP}^2$.

Finally, consider the solutions $(2,t,t)$ with $t\geq 2$. 
Suppose that $\widetilde\ga\subset T$ is a simple closed curve such that $(\ga_1,\widetilde\ga)$ is a symplectic basis. 
Then, $\ga_2=a\ga_1+t\widetilde\ga$, $\ga_3 = b\ga_1+t\widetilde\ga$ for some $a,b\in\Z$ coprime with $t$, and $2 = \ga_2\cdot\ga_3 = t(a-b)$ implies $t=2$ and $a=b+1$. 
But $b+1$ and $b$ cannot be both odd, so $(2,t,t)$ is never realized.

\subsection{Minimal solutions when $(\de_1,\de_2,\de_3)=(-1,1,1)$}\label{ss:-111}

In this section we determine the weakly minimal solutions of the equation  
\begin{equation}\label{e:negcohn}
x^2 + xyz -a = y^2 + z^2,
\end{equation}
where $a \in \{0,4\}$.

\begin{prop}\label{p:negcohn}
	Let $(x,y,z)$ be a solution of Equation~\eqref{e:negcohn} with $|x|$, $|y|$, $|z|$ pairwise distinct.  
	Then, $(x,y,z)$ is not weakly minimal. Indeed, mutation at the greatest element in absolute value makes the quantity $|x|+|y|+|z|$ strictly smaller. Moreover, if $\min(|x|,|y|,|z|)>1$ then mutation at the other 
	two elements makes the quantity $|x|+|y|+|z|$ strictly bigger. 
\end{prop}

\begin{proof} 
Note that if $(x,y,z)$ is a solution then so is $(x,z,y)$. Therefore, 
up to swapping $y$ and $z$ we may assume $|y|>|z|$. Moreover, it is easy to check that $xyz\neq 0$, 
because otherwise at least two among $x,y,z$ would be zero. We split the proof into the three cases
$|x| > |y| > |z|$, $|y| > |x| > |z|$ and $|y| > |z| > |x|$.

\smallskip\noindent
{\bf First case:} $|x| > |y| > |z|$. We claim that $|\hat x|<|x|$. Up to sign changes it suffices to assume $x,y>0$. This implies $z<0$, because $z>0$ gives the contradiction
\[
y^2+z^2 = x^2+xyz-a > y^2 + 2xz-a \ge y^2+2(z+2)z-a> y^2+z^2.
\]
We have $x(yz+x)=y^2+z^2+a>0$, therefore $\hat x = -yz - x < 0$, which implies, since $-yz>0$ and $-x<0$, $|\hat x|<|x|$. 
This shows that mutation at the greatest element in absolute value makes the quantity $|x|+|y|+|z|$ strictly smaller, in particular that $(x,y,z)$ is not weakly minimal. 
To prove the last part of the statement suppose $|z|>1$. 
Then, $y>2$ and  
\begin{equation}\label{e:zhatineq}
2|z| < xy < xy-z + |z| = \hat z + |z|,
\end{equation}
therefore $|z| < \hat z$.
\color{black}
Similarly, we have 
\[
2 x < |xz| \leq |xz - y| + |y| < |\hat y| + x,
\]
therefore $y < x < |\hat y|$. This concludes the proof of the statement in the first case. 

Before tackling the other cases we observe that up to sign changes we may 
assume $y>0$. Moreover, if $xz<0$ we have  
\[
y^2 -xyz + z^2 > x^2 -xyz + z^2 > x^2 \ge x^2-a,
\]
which is a contradiction. Hence, after a further sign change we may assume 
$x>0$ and $z>0$ as well. Therefore, from now on we assume $x,y,z>0$. 

\smallskip\noindent
{\bf Second case:} $|y|>|x|>|z|$. 
Note that $y\hat y = z^2-x^2+a \geq 0$ forces $x=2$, $z=1$ and $a=4$, in which case   
$y\hat y = 1$, which is impossible because $y>z\geq 1$. Therefore we have 
$y\hat y<0$, which implies $\hat y = xz-y<0$. 
It follows that $0 < xz < y$, hence $|\hat y| < |y|$. 
To prove the last part of the statement suppose $|z|>1$. Then, 
$|z|<|\hat z|$ follows as before from~\eqref{e:zhatineq}. 
Finally, we have 
$$
2|y| \leq |-yz - x| + |x| < |\hat x| + |y|,
$$	
hence $|x|<|y|<|\hat x|$. This proves the statement in the second case. 

	\smallskip\noindent
	{\bf Third case:} $|y|>|z|>|x|$.
Since $y\hat y = z^2-x^2+a>0$ we have $\hat y = xz-y>0$. If $\hat y \ge y$ then $xz \ge 2y$ 
and, using that $z\geq x+1\geq 2$, we get the contradictory inequalities
\[
y^2 + z^2 = x^2 + xyz-a > 2y^2-a \geq y^2 + (|z|+1)^2 - a > y^2 + z^2.
\]
Therefore $|y|=y > \hat y = |\hat y|$. 
To prove the last part of the statement observe that 
\[
2|x| < |yz| \leq |-yz-x|+|x| = |\hat x| + |x|, 
\]
hence $|x|<|\hat x|$. Using the extra assumption $|x|\geq 2$ we have  
$$
2|y| \leq |xy - z| + |z| < |\hat z| + |y|,
$$
hence $|z|<|y|<|\hat z|$. This proves the statement in the third case and concludes the proof of the proposition. 
\end{proof}
\color{black}

\begin{lemma}\label{l:negcohn}
	Let $(x,y,z)$ be a weakly minimal solution of Equation~\eqref{e:negcohn}. 
	Then, $(|x|,|y|,|z|)$ is of the form 
\begin{itemize}
    \item $(t,t,0)$ or $(t,0,t)$ for some $t\geq 0$ if $a=0$;
    \item $(2,t,t)$ for some $t\geq 0$ if $a=4$.
\end{itemize}
\end{lemma} 

\begin{proof}
By Proposition~\ref{p:negcohn}, $|x|,|y|,|z|$ are not pairwise distinct: therefore we can assume, possibly after changing two signs, that two variables are equal and nonnegative; also, since the equation is symmetric in $y$ and $z$, it suffices to discuss the cases $x=y$ and $y=z$.
We have the following possibilities:
\begin{itemize}
\item 
$a=0$ and $y=z\ge 0$: we have $x^2=(2-x)y^2$.
If $y=0$ we get the solution $(0,0,0)$, otherwise we must have $x=cy$ with $c^2=2-x=2-cy$, so that 
$c\ |\ 2$: the values of $c$ for which $y>0$ are $1$ and $-2$, and the corresponding solutions are $(1,1,1)$ and $(-2,1,1)$, none of which is weakly minimal.
\item 
$a=0$ and $x=y \ge 0$: we have $x^2z=z^2$, so that either $z=0$ or $z=x^2$, and the weakly minimal solutions are those of the form $(t,t,0)$ for some $t \ge 0$. Because of the symmetry in $y$ and $z$ this also gives 
the weakly minimal solutions $(t,0,t)$, $t\ge 0$. 
\item 
$a=4$ and $y=z \ge 0$: we have $x^2-4=(2-x)y^2$, so that either $x=2$ or $x=-y^2-2$, and the weakly minimal solutions are those of the form $(2,t,t)$ for some $t \ge 0$.
\item 
$a=4$ and $x=y \ge 0$: we have $x^2z-4=z^2$, which implies $z>0$ and $z\ |\ 4$. The only case where $x$ can be an integer is $z=2$, which gives the weakly minimal solution $(2,2,2)$.
\end{itemize}
\end{proof}

\subsection{Underlying $4$-manifolds when 
$(\de_1,\de_2,\de_3)=(-1,1,1)$ and $a=0$}\label{ss:-111-a=0} 

When $t=0$ the minimal solution $(t,t,0)=(t,0,t)=(0,0,0)$ of Lemma~\ref{l:negcohn}
comes from the trivial factorization $\tau_\la  = \tau_\la \tau_\la \tau_\la^{-1}$, 
because $0 = \ga_1\cdot\ga_2=\ga_2\cdot\ga_3=\ga_1\cdot\ga_3$ and $m_L(\la)=\pm\la$ 
implies $\tau_{\ga_i}=\tau_\la$ for each $i$. It is easy to check that in this case 
$\hat X \cong \CP^2 \# 2\overline{\CP}^2 \# S^1\x S^3$. When $t\neq 0$ the same kind of  
analysis shows that the solution $(t,t,0)$ comes from the factorization  
$\tau_\la = \tau_\la \tau_\ga\tau_\ga^{-1}$. Similarly, $(t,0,t)$ 
comes from a factorization of the form 
\[
\tau_\la = \tau_\ga\tau_{\ga'}\tau_\ga^{-1} = 
\tau_{\tau_{\ga}(\ga')}\tau_\ga\tau_\ga^{-1} = 
\tau_\la\tau_\ga\tau_\ga^{-1}.
\]
Since Hurwitz moves on the factorization do not change the underlying $4$-manifold, 
it suffices to analyze the former case. We may assume $\ga=t \mu + q \la$ 
for some integer $q$ coprime with $t$, with $q=0$ if $t=1$. 
Now we can proceed as in the case $n=0$ in Section~\ref{s:chileq2}. In fact, as in that 
case the components $\ga_1=\ga_2=\ga$ of $L$ cobound an annulus $A$ of 
the form $\ga\x [0,1]$ with framings with respect to $A$ opposite to each other and equal to $1$ 
in absolute value. Sliding $\ga_1$ over $\ga_2$ changes $L$ into the framed link 
consisting of $\ga_2$ and a $0$-framed meridian of $\ga_2$. 
Applying~\cite[Proposition~3.1]{LP22-2} (and the preceding discussion in the same paper) it follows that $\hat X\cong\overline{\CP}^2\# B_{t,q}\cup -B_{t,q}$. 
Since $B_{t,q}\cup -B_{t,q}$ is not simply connected when $t>1$ while $B_{1,0}\cup -B_{1,0}\cong S^4$, we get $\hat X\cong\overline\CP^2$ precisely when the factorization is $\tau_\la\tau_\mu\tau_\mu^{-1}$. 

\subsection{Underlying $4$-manifolds when 
$(\de_1,\de_2,\de_3)=(-1,1,1)$ and $a=4$}\label{ss:-111-a=4} 
The argument at the beginning of Section~\ref{ss:111-a=4} shows that $(2,0,0)$ cannot be realized, while the argument at the end of the same section shows that $(2,t,t)$ is not realizable if $t\geq 2$. 
Therefore, we only need to discuss the solution $(2,1,1)$. 
Since $\ga_1\cdot\ga_2=1$ we may take $\ga_1$ and $\ga_2$ as a symplectic basis and write $\ga_3=a\ga_1+b\ga_2$.
Then, $\ga_1\cdot\ga_3=1$ and $\ga_2\cdot\ga_3=2$ imply $\ga_3=\ga_2-2\ga_1$ 
and the monodromy is $m_L = \tau_{\ga_2-2\ga_1}\tau_{\ga_2}\tau_{\ga_1}^{-1}$. 
One can easily check that $m_L(\ga_1)=-\ga_1$, which implies $\ga_1=\pm\la$. 
Moreover, $\ga_2 = \mp\mu+k\la$ for some $k\in\Z$ and $m_L(\ga_2)=5\ga_1-\ga_2$.
Arguing as in Section~\ref{ss:111-a=0}, we may apply a $1$-handle twist which 
does not change the underlying $4$-manifold. Therefore, up to changing the signs 
of all the curves, we may assume $\ga_1=-\la$ and 
$\ga_2 = \mu$. Thus, $\ga_3 = \mu + 2\la$, $m_L = -\tau_\la^{-5}$. 
As in the previous cases, Figure~\ref{f:fig3} illustrates some simple Kirby calculus showing $\hat X\cong \CP^2$.   
\begin{figure}[ht]
	\labellist
	\hair 2pt
	\pinlabel $+1$ at 17 17
	\pinlabel $-1$ at 58 44
	\pinlabel $+1$ at 70 65
	\pinlabel $\ga_2$ at 86 30
	\pinlabel $\ga_1$ at 109 30
	\pinlabel $\ga_3$ at 127 73
	\pinlabel $+1$ at 180 23
	\pinlabel $+1$ at 223 68
	\pinlabel $0$ at 215 43
	
	\pinlabel $0$ at 310 50
	\pinlabel $+1$ at 310 80
	\endlabellist
	\centering
	\includegraphics[width=11cm]{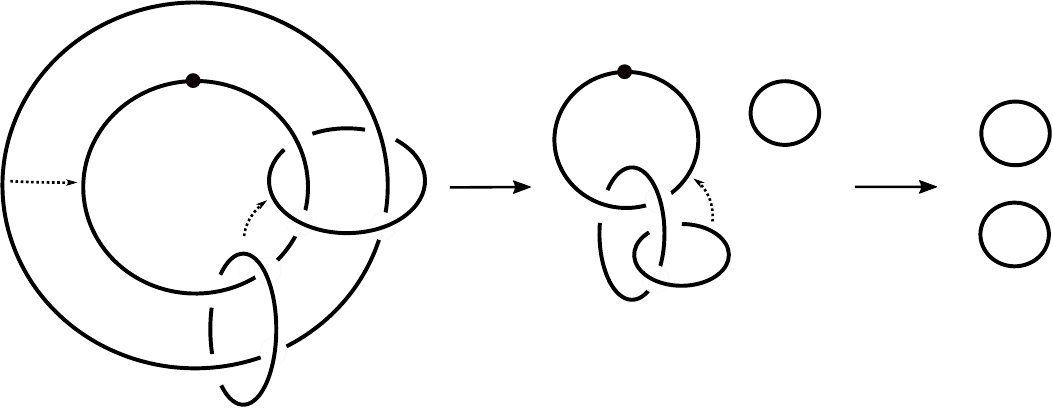}
	\caption{Proof that $\hat X\cong\CP^2$ when $(\de_1,\de_2,\de_3)=(-1,1,1)$ and $a=4$}
	\label{f:fig3}
\end{figure}

Putting together the results of Sections~\ref{ss:111-a=0},~\ref{ss:111-a=4},~\ref{ss:-111-a=0} and~\ref{ss:-111-a=4} yields Theorem~\ref{t:chi=3} and establishes Theorem~\ref{t:type11} when $\chi(\hat X)=3$. 

\section{Proof of Theorem~\ref{t:disj-emb}}\label{s:disj-emb}

Suppose that we have a $(1,1,3,1)$-decomposition of $\CP^2$ with factorization 
\begin{equation}\label{e:*}
\tau^{\de_3}_{\gamma_3} \tau^{\de_2}_{\gamma_2} \tau^{\de_1}_{\gamma_1}=\pm \tau_\la^k 
\end{equation} 
and associated triple $(x_1,x_2,x_3)$, defined as in Equation~\eqref{e:relations1}: 
\begin{equation*}
(x_1,x_2,x_3) = (\ga_2\cdot\ga_3, \ga_1\cdot\ga_3,\ga_1\cdot\ga_2) 
\end{equation*}
In this section we are going to show how to recover the curves $\ga_i=p_i \mu + q_i \la$ from $(x_1,x_2,x_3)$. The key fact we are going to use is the following. 
\begin{lemma}\label{l:disballs}
A horizontal decomposition of a smooth, closed $4$-manifold $X$ of type $(1,1,3,1)$ with factorization~\eqref{e:*} determines an embedding of the disjoint union $\de_1 B_{p_1,q_1}\cup\de_2 B_{p_2,q_2}\cup\de_3 B_{p_3,q_3}$ into $X$. 
\end{lemma}

\begin{proof}
Note that $S^1\x D^3$ is obtained from $H_1\x [0,1]$ by smoothing corners, where $H_1$ is a genus-$1$ handlebody. 
Hence, the attaching spheres of the three horizontal $2$-handles of $X$ are of the form 
$\ga_i = \ga\x\{t_i\}\subset\TT_i:=\del H_1\x\{t_i\}$ for some $t_1, t_2, t_3\in (0,1)$ with $t_1<t_2<t_3$. 
For some small $\ep>0$ the subsets $H_1\x [t_i-\ep,t_i+\ep]\subset H_1\x [0,1]$, $i=1,2,3$, 
are pairwise disjoint, and smoothing corners yields three pairwise disjoint copies of $S^1\x D^3$ sitting inside $S^1\x D^3$.
Clearly, the core disks of the three horizontal $2$-handles are pairwise disjoint as well, therefore $X$ contains the disjoint union 
$X_1\cup X_2\cup X_3$, where $X_i$ is the smooth $4$-manifold obtained by attaching a four-dimensional $2$-handle to $S^1\x D^3$ 
along $\ga_i\subset\TT_i$ with framing $\fr(\ga_i)=\fr_{\TT_i}(\ga_i) - \de_i$. 
By ~\cite[Proposition~3.2]{LP22-2}, $X_i$ is orientation-preserving diffeomorphic to $\de_i B_{p_i,q_i}$. 
This concludes the proof. 
\end{proof}

Evaluating the left-hand side of~\eqref{e:*} at $\ga_1$ and taking the products with $\ga_i$ gives 
\[
\begin{split}
\tau^{\de_3}_{\gamma_3} \tau^{\de_2}_{\gamma_2} \tau^{\de_1}_{\gamma_1}(\ga_1) \cdot \ga_i=
\tau^{\de_3}_{\gamma_3} \tau^{\de_2}_{\gamma_2}(\ga_1) \cdot \ga_i=
\tau^{\de_2}_{\gamma_2}(\ga_1) \cdot \tau^{-\de_3}_{\gamma_3}(\ga_i)=
(\ga_1-\de_2 x_3 \ga_2) \cdot (\ga_i-\de_3 (\ga_3 \cdot \ga_i) \ga_3)
\end{split} 
\]
for $i=1,2,3$. When $i=1$, using that $(x_1,x_2,x_3) \in S_a^{\de_1,\de_2,\de_3}$ where $a=2-\tr (m_L) \in \{0,4\}$
we obtain the equality  
\[
\tau^{\de_3}_{\gamma_3} \tau^{\de_2}_{\gamma_2} \tau^{\de_1}_{\gamma_1}(\ga_1) \cdot \ga_1=
\de_3x_2^2+\de_2 x_3^2-\de_2 \de_3 x_1x_2x_3=\de_1 a-\de_1\de_2\de_3x_1^2.
\]
On the other hand, evaluating the right-hand side of~\eqref{e:*} at $\ga_1$ and taking the intersection product with $\ga_i$ gives \[
\pm \tau_\la^k(\ga_1) \cdot \ga_i=\pm(\ga_1+k(\la \cdot \ga_1) \la) \cdot \ga_i.
\]
Equating the expressions obtained for $i=1,2,3$ we get following equations: 
\begin{equation}\label{e:xp-equations}
\begin{cases}
\de_1 a-\de_1\de_2\de_3x_1^2 = \pm k p_1^2\\
x_3+\de_3x_1x_2-\de_2\de_3x_1^2x_3 = \pm (x_3+kp_1p_2)\\
x_2-\de_2 x_1 x_3 = \pm (x_2+kp_1p_3).
\end{cases}
\end{equation}

Note that the $\pm$ signs on the right-hand sides are plus signs if and only if $a=0$, since both statements are equivalent to $\tr(m_L)=2$. 

\begin{rmk}\label{r:factperm}
If a $(1,1,3,1)$-decomposition of $\pm\CP^2$ has factorization $(\tau_{\ga_3}^{\de_3},\tau_{\ga_2}^{\de_2},\tau_{\ga_1}^{\de_1})$, the sequence of Hurwitz moves 
\[
\tau_{\ga_3}^{\de_3}\tau_{\ga_2}^{\de_2}\tau_{\ga_1}^{\de_1} \rightarrow
\tau_{\ga_3}^{\de_3}\tau_{\tau_{\ga_2}^{\de_2}(\ga_1)}^{\de_1}\tau_{\ga_2}^{\de_2}
\rightarrow
\tau_{\tau_{\ga_3}^{\de_3}\tau_{\ga_2}^{\de_2}(\ga_1)}^{\de_1}\tau_{\ga_3}^{\de_3}\tau_{\ga_2}^{\de_2}=
\tau_{\tau_{\ga_3}^{\de_3}\tau_{\ga_2}^{\de_2}\tau_{\ga_1}^{\de_1}(\ga_1)}^{\de_1}\tau_{\ga_3}^{\de_3}\tau_{\ga_2}^{\de_2}=
\tau_{\pm \tau_{\la}^k(\ga_1)}^{\de_1}\tau_{\ga_3}^{\de_3}\tau_{\ga_2}^{\de_2}
\]
produces another Kirby diagram of $\pm\CP^2$ and an embedding of the same triple of rational homology 
balls, because changing a curve $\ga_i$ in a factorization by a power of $\tau_\la$ corresponds 
to twisting the $1-$handle of $S^1 \x D^3 \cup X_{\ga_i}$. Note that these moves permute 
$(\de_1,\de_2,\de_3)$ cyclically. Therefore, in order to prove Theorem~\ref{t:disj-emb} it suffices to 
consider, for each possible value of $a$, a single 
representative of $\bde=(\de_1,\de_2,\de_3)$ in each orbit of the $\Z/3\Z$-action by cyclic permutations.
\end{rmk} 

The proof of Theorem~\ref{t:disj-emb} proceeds as follows. The analysis of Sections~\ref{ss:111-a=0}, \ref{ss:111-a=4}, 
\ref{ss:-111-a=0} and~\ref{ss:-111-a=4} together with Remark~\ref{r:factperm} show that 
a closed $4$-manifold with a $(1,1,3,1)$-decomposition is diffeomorphic to $\CP^2$ or $\overline\CP^2$ precisely 
in the cases of Table~\ref{tab:table}. In Sections~\ref{ss:pairs=111-a=0}, \ref{ss:pairs=111-a=4}, \ref{ss:pairs=-111-a=0} and~\ref{ss:pairs=-111-a=4} we analize the cases corresponding to the four columns of Table~\ref{tab:table}. It turns out that the second column of Table~\ref{tab:table} is not relevant for Theorem~\ref{t:disj-emb}, while the first, third and fourth columns correspond, respectively, to the embeddings of Theorem~\ref{t:disj-emb}(1), (2) and (3). 
\begin{table}[ht]
\centering
\begin{tabular}{ c|cccc } 
$\bde$ &  $(1,1,1)$ & $(1,1,1)$ & $(-1,1,1)$ & $(-1, 1, 1)$ \\ 
\hline 
$a$ & $0$ & $4$ & $0$ & $4$ \\
$m_L$ & $\tau_\la^{-9}$ & $-\tau_\la^{-3}$ & $\tau_\la$ & $-\tau_\la^{-5}$\\
minimal factorization & 
$\tau_{\mu+2\la}\tau_{2\mu+\la}\tau_{\mu-\la}$ & 
$\tau_{\mu+\la}\tau_\la\tau_\mu$ & 
$\tau_\la\tau_\mu\tau_\mu^{-1}$ 
& $\tau_{\mu+2\la}\tau_\mu\tau_{-\la}^{-1}$\\  
minimal solutions & $(3,3,3)$ & $(-1,1,1)$ & $(1,1,0)$, $(1,0,1)$ & $(2,1,1)$\\
underlying $4$-manifold & $\CP^2$ & $\overline\CP^2$ & $\overline\CP^2$ & $\CP^2$\\
\end{tabular}
\caption{}
\label{tab:table}
\end{table}

\subsection{$\bde=(1,1,1)$, $a=0$ and $m_L=\tau_\la^{-9}$}\label{ss:pairs=111-a=0}
In this case Equations~\eqref{e:xp-equations} 
take the form 
\begin{equation}\label{e:relations0}
\begin{cases}
-x_1^2 = k p_1^2 \\
x_3 + x_1 x_2 - x_1^2 x_3 = x_3 + k p_1 p_2\\
x_2 - x_1 x_3 = x_2 + k p_1 p_3.
\end{cases}
\end{equation}
Moreover, in Section~\ref{ss:111-a=0} we showed that the underlying $4$-manifold is $\CP^2$ 
only when $k=-9$ and $(x_1,x_2,x_3)$ is a non-trivial solution of Equation~\eqref{e:cohn} 
obtained via a sequence of mutations and sign changes from the minimal solution $(3,3,3)$. 
Therefore -- as one can prove by induction on the number of mutations -- we can write $x_i=3y_i$, 
where $(y_1,y_2,y_3)$ are nonzero integers satisfying Markov's equation
\[
y_1^2+y_2^2+y_3^2 = 3 y_1 y_2 y_3.
\]
Relations~\eqref{e:relations0} can be written as follows:  
\begin{equation}\label{e:relations2}
\begin{cases}
y_1^2 = p_1^2 \\
y_1 (y_2 - 3 y_1 y_3) = - p_1 p_2\\
y_1 y_3 = p_1 p_3.
\end{cases}
\end{equation}
Solving for $p_1$, $p_2$ and $p_3$ we obtain 
$p_1 = \pm y_1$, $p_2 = \pm (3y_1 y_3 - y_2) = \pm \hat y_2 $ and $p_3 = \pm y_3$.  
Therefore we have 
\[
\ga_1 = \pm y_1\mu + q_1\la,\quad 
\ga_2 = \pm \hat y_2\mu + q_2\la\quad\text{and}\quad
\ga_3 = \pm y_3\mu + q_3\la,  
\]
where, due to Relations~\eqref{e:relations1}, the $q_i$'s satisfy 
\begin{equation}\label{e:qsystem}
y_1 q_2 - \hat y_2 q_1 = \pm 3y_3,\quad 
y_1 q_3 - y_3 q_1 = \pm 3y_2\quad\text{and}\quad 
\hat y_2 q_3 - y_3 q_2 = \pm 3y_1.
\end{equation} 
Note that the first and last of Equations~\eqref{e:qsystem} imply 
\begin{equation}\label{e:qcong}
q_1 \equiv \mp 3\frac{y_3}{\hat y_2}\bmod y_1,\quad 
q_2 \equiv \mp 3\frac{y_1}{y_3}\bmod\hat y_2\quad\text{and}\quad 
q_3 \equiv \pm 3\frac{y_1}{\hat y_2}\bmod y_3.
\end{equation}
Therefore, arguing as in~\cite[Proposition~3.2 and Lemma~5.1]{LP22-2} we conclude that when a $(1,1,3,1)$-decomposition induces an embedding in $\CP^2$ of a  disjoint union of $B_{y,q}$'s, then union is of the form 
\begin{equation}\label{e:tripleofballs-111}
B_{y_1,q_1}\cup B_{\hat y_2,q_2}\cup B_{y_3,q_3},
\end{equation}
where the $y_i$'s satisfy Markov's equation and the $q_i$'s satisfy Equations~\eqref{e:qcong}.

Conversely, given any nonzero Markov triple 
$(y_1,y_2,y_3)$, the system of 
congruences~\eqref{e:qcong} has a unique solution in the $q_i$'s. In fact, when 
$(y_1,y_2,y_3)$ is a nonzero Markov triple, for either choice of sign 
there is a triple $(q_1,q_2,q_3)$ satisfying 
Equations~\eqref{e:qsystem} and therefore also Equations~\eqref{e:qcong}. Indeed, 
since $(y_1,\hat y_2,y_3)$ is also a nonzero Markov triple, $y_1$, $\hat y_2$ and $y_3$ are 
pairwise coprime and $y_2 \hat y_2=y_1^2+y_3^2$, which implies 
$y_3/\hat y_2\equiv y_2/y_3\bmod y_1$. 
Therefore we can find $q_1$, $q_2$ and $q_3$ which solve simultaneously 
the first two Equations~\eqref{e:qsystem}. 
But notice that the last equation is obtained dividing by $y_1$ the difference between 
the second equation multiplied by $\hat y_2$ 
and the first equation multiplied by $y_3$. 
It follows that $(q_1,q_2,q_3)$ is in fact a solution 
the whole System~\eqref{e:qsystem}. Now define $\ga_1 = y_1\mu+q_1\la$, 
$\ga_2 = \hat y_2\mu+q_2\la$ and $\ga_3 = y_3\mu+q_3\la$ and consider the  
cobordism $X\co\del_- X\to\del_+ X$ with a horizontal decomposition of type $(1,1,3,1)$ prescribed  
by the factorization $(\tau_{\ga_3},\tau_{\ga_2},\tau_{\ga_1})$. Then, by 
Equations~\eqref{e:qsystem} the triple 
$(x_1,x_2,x_3) = (\ga_2\cdot\ga_3, \ga_1\cdot\ga_3,\ga_1\cdot\ga_2)$ 
is a nonzero solution of Equation~\eqref{e:cohn} with $a=0$. 
Using the fact that $(y_1,y_2,y_3)$ and $(y_1,\hat y_2,y_3)$ 
are Markov triples mutant of each other it is straightforward to check that 
Equation~\eqref{e:eq2} holds replacing $(p_1,p_2,p_3)$ with $(y_1,\hat y_2,y_3)$ 
and $(x_1,x_2,x_3)$ with $(3y_1,3y_2,3y_3)$. Since Equation~\eqref{e:eq2} is equivalent 
to $m_L(\la)\cdot\la=0$, this implies $m_L(\la) = \pm \la$ and therefore $X\co S^3\to S^3$.
Moreover, by Lemma~\ref{l:cohn} the triple $(x_1,x_2,x_3)$ is obtained by a sequence 
of mutations from $(3,3,3)$, the arguments of Section~\ref{ss:111-a=0} show that 
$\hat X\cong\CP^2$ and the disjoint union~\eqref{e:tripleofballs-111} embeds in it. 
Since every nonzero Markov triple is of the form $(y_1,\hat y_2,y_3)$, 
this shows that any disjoint union of the form~\eqref{e:tripleofballs-111} embeds in $\CP^2$. 
This covers the embeddings of Theorem~\ref{t:disj-emb}(1).

\subsection{$\bde = (1,1,1)$, $a=4$ and $m_L = -\tau_\la^{-3}$}\label{ss:pairs=111-a=4}

By the results of Section~\ref{ss:111-a=4}, in this case the underlying $4$-manifold is $\overline\CP^2$. 
Moreover, the only realizable weakly minimal triple in $S_4^{1,1,1}$ is $(-1,1,1)$, and a simple verification shows that  
the only triples obtainable from $(-1,1,1)$ via mutations are $(2,1,1)$, $(-1,-2,1)$ and 
$(-1,1,-2)$. On the other hand, as observed in Section~\ref{ss:hurwitz}, mirroring the   
horizontal decompositions produces decompositions of $\CP^2$ with monodromy 
$m_L = -\tau_\la^3$ and it induces an identification 
$S_4^{-1,-1,-1}\longleftrightarrow S_4^{1,1,1}$ given by $(x_1,x_2,x_3)\mapsto (-x_1,-x_2,-x_3)$. Thus, since we are interested in embeddings in $\CP^2$, we only need to consider the triples $(1,-1,-1)$, $(-2,-1,-1)$, $(1,2,-1)$ and $(1,-1,2)$. When $\de=(-1,-1,-1)$, $a=4$ and $m_L = -\tau_\la^3$   Equations~\eqref{e:xp-equations} become 
\begin{equation}\label{e:relations3}
\begin{cases}
x_1^2 - 4 = -3 p_1^2 \\
2x_3 - x_1 x_2 - x_1^2 x_3 = - 3 p_1 p_2 \\
2 x_2 + x_1 x_3 = - 3 p_1 p_3.
\end{cases}
\end{equation}
If $(x_1,x_2,x_3)=(1,-1,-1)$ it is easy to check that the only solutions are  
$(p_1,p_2,p_3) = \pm (1,0,1)$. When $(x_1,x_2,x_3)=(1,2,-1)$ and 
$(x_1,x_2,x_3)=(1,-1,2)$ the only solutions are, respectively, $(p_1,p_2,p_3)=\pm(1,1,-1)$ 
and $(p_1,p_2,p_3)=\pm (1,-1,0)$. Finally, when $(x_1,x_2,x_3)=(-2,-1,-1)$ 
Equations~\eqref{e:relations3} imply $p_1=0$ and say nothing about $p_2$ nor $p_3$. 
In order to extract information about $p_2$ and $p_3$ we evaluate both sides of Equation~\eqref{e:*} on $\ga_2$ and we take the intersection product with $\ga_3$, obtaining the equation
\[
2 x_1 - x_2 x_3 - x_1 x_3^2  = -3 p_2 p_3.
\]
The only solutions when $(x_1,x_2,x_3)=(-2,-1,-1)$ are $(p_2,p_3)=\pm (1,1)$, 
and we conclude $(p_1,p_2,p_3) = \pm (0,1,1)$. We observe that in every case we have 
$\max\{p_i\}\leq 1$, therefore none of the embeddings obtained when $\bde = (1,1,1)$ and 
$a=4$ are relevant for the statement of Theorem~\ref{t:disj-emb}. 

\subsection{$\bde = (-1,1,1)$, $a=0$ and $m_L = \tau_\la$}\label{ss:pairs=-111-a=0}
The analysis of Section~\ref{ss:-111-a=0} shows that 
$(x_1,x_2,x_3)$ is obtained by a sequence of mutations and sign changes from 
either $(1,1,0)$ or $(1,0,1)$. This implies that the $x_i$'s are pairwise coprime, 
otherwise by Equation~\eqref{e:negcohn} they would have a non-trivial common factor, 
but such property is preserved by mutation and sign changes while $(1,1,0)$ and 
$(1,0,1)$ do not have it. Moreover, the only solution of Equation~\eqref{e:negcohn} 
with $x_1=0$ is the trivial one $(0,0,0)$, which is equivalent only to itself  
up to mutations and sign changes. Therefore $x_1\neq 0$.

Since the underlying $4$-manifold is $\overline\CP^2$ we argue as in 
Section~\ref{ss:pairs=111-a=4}. There is a $1-1$ correspondence $S_0^{1,-1,-1}\longleftrightarrow S_0^{-1,1,1}$ given by the map sending 
$(x_1,x_2,x_3)$ to $(-x_1,-x_2,-x_3)$, and induced by the operation of taking 
the mirror image of the Kirby diagram of the underlying $4$-manifold.
Thus, the triples $(x_1,x_2,x_3)\in S_0^{1,-1,-1}$ whose underlying $4$-manifold is $\CP^2$ are those obtained by a sequence of mutations and sign changes from either 
$(-1,-1,0)$ or $(-1,0,-1)$, the $x_i$'s are pairwise coprime and $x_1\neq 0$. 
When $\bde = (1,-1,-1)$, $a=0$ and $k=-1$ Equations~\eqref{e:xp-equations} read  
\begin{equation}\label{e:relations4}
\begin{cases}
x_1^2 = p_1^2, \\
x_1 (x_2 + x_1 x_3) = p_1 p_2 \\
- x_1 x_3 = p_1 p_3.
\end{cases}
\end{equation}
and  solving for $p_1$, $p_2$ and $p_3$ we obtain 
$p_1 = \pm x_1$, $p_2 =\mp\hat x_2$ and $p_3=\mp x_3$. 
Therefore 
\[
\ga_1 = \pm x_1\mu + q_1\la,\quad 
\ga_2 = \mp \hat x_2\mu + q_2\la\quad\text{and}\quad
\ga_3 = \mp x_3\mu + q_3\la  
\]
and, due to Relations~\eqref{e:relations1}, the $q_i$'s satisfy 
\begin{equation}\label{e:qsystem2}
x_1 q_2 + \hat x_2 q_1 = \pm x_3,\quad 
x_1 q_3 + x_3 q_1 = \pm x_2\quad\text{and}\quad 
x_3 q_2 - \hat x_2 q_3 = \pm x_1.
\end{equation} 
Note that the first and last of Equations~\eqref{e:qsystem2} 
imply 
\begin{equation}\label{e:qcong2}
\hat x_2 q_1\equiv\pm x_3\bmod x_1,\quad 
x_3 q_2\equiv\pm x_1\bmod\hat x_2,\quad\text{and}\quad  
x_2 q_3\equiv\mp x_1\bmod x_3.
\end{equation}
We can conclude that if a $(1,1,3,1)$-decomposition with factorization of the form 
$(\tau_{\ga_3}^{-1},\tau_{\ga_2}^{-1},\tau_{\ga_1})$ 
induces an orientation-preserving embedding in $\CP^2$ of a 
disjoint union of $\pm B_{x,q}$'s, then such a union is of the form 
\begin{equation}\label{e:tripleofballs-1-1-1}
B_{x_1,q_1}\cup -B_{x_2,q_2}\cup -B_{x_3,q_3},
\end{equation}
where the $x_i$'s satisfy Equation~\eqref{e:eq1} with $\bde = (1,-1,-1)$ and $a=0$ and the $q_i$'s 
satisfy~\eqref{e:qcong2}. 

We are now going to show that, conversely, if 
$(x_1,x_2,x_3)$ is a non-zero triple of integers satisfying 
Equation~\eqref{e:eq1} with $\bde = (1,-1,-1)$ and $a=0$ and the triple $(x_1,x_2,x_3)$  
is obtained by a sequence of mutations and sign changes from either 
$(-1,-1,0)$ or $(-1,0,-1)$, then the disjoint union~\eqref{e:tripleofballs-1-1-1}, 
where the triple $(q_1,q_2,q_3)$ satisfies~\eqref{e:qcong2}, smoothly embeds in $\CP^2$.
We start observing that there is a triple $(q_1,q_2,q_3)$ satisfying 
Equations~\eqref{e:qsystem2} with a plus sign on the right-hand side. 
Indeed, since $(x_1,\hat x_2,x_3)$ is also a solution of Equation~\eqref{e:eq1} with 
$\bde = (1,-1,-1)$ and $a=0$, by the argument given at the beginning of this section 
$x_1$, $\hat x_2$ and $x_3$ are pairwise coprime. Moreover, 
by Equation~\eqref{e:eq1} we have $x_2 \hat x_2=x_3^2-x_1^2$, which implies 
$x_2 \hat x_2\equiv x_3^2\bmod x_1$ and, since $x_1\neq 0$, $(\hat x_2, x_3)\neq (0,0)$.  
Therefore there is a simultaneous solution $\xi\bmod x_1$ of the equations 
\[
\xi\hat x_2 = \pm x_3\bmod x_1\quad\text{and}\quad\xi x_3 = \pm x_2\bmod x_1.
\]
This guarantees that we can find $q_1$, $q_2$ and $q_3$ 
solving simultaneously the first two Equations~\eqref{e:qsystem2} with a plus sign 
on the right-hand side. Now note that the last equation with a plus sign on the right-hand 
side is obtained dividing by $x_1$ the difference 
between the first equation multiplied by $x_3$ and the second equation multiplied 
by $\hat x_2$. Since $x_1\neq 0$, it follows that $(q_1,q_2,q_3)$ is in fact 
a solution of the whole system. Now define 
\[
\ga_1 = x_1\mu+q_1\la,\quad \ga_2 = \hat x_2\mu+q_2\la\quad\text{and}\quad
\ga_3 = x_3\mu+q_3\la 
\]
and consider the cobordism $X\co\del_- X\to\del_+ X$ with a 
horizontal decomposition of type $(1,1,3,1)$ prescribed  
by the factorization $(\tau_{\ga_3}^{-1},\tau_{\ga_2}^{-1},\tau_{\ga_1})$. 
Then, by Equations~\eqref{e:qsystem2} the triple 
\[
(\ga_2\cdot\ga_3, \ga_1\cdot\ga_3,\ga_1\cdot\ga_2) = (x_1,x_2,x_3)
\]
is a solution of Equation~\eqref{e:eq1} with $\bde = (1,-1,-1)$ and $a=0$.
Using the fact that $(x_1,x_2,x_3)$ and $(x_1,\hat x_2,x_3)$ 
are mutant of each other it is straightforward to check using the equality 
$x_2\hat x_2 = x_3^2-x_1^2$ that Equation~\eqref{e:eq2} holds for $\bde = (1,-1,-1)$ replacing $(p_1,p_2,p_3)$ with $(x_1,\hat x_2,x_3)$. 
Since Equation~\eqref{e:eq2} is equivalent to $m_L(\la)\cdot\la=0$, 
this implies $m_L(\la) = \pm \la$ and therefore $X\co S^3\to S^3$. 
But the analysis of Section~\ref{ss:-111-a=0} shows that the $4$-manifold underlying 
the triple $(-x_1,-x_2,-x_3)\in S^{-1,1,1}_0$ is $\overline{\CP}^2$ precisely 
when $(-x_1,-x_2,-x_3)$ is obtained by a sequence of mutations from either 
$(1,1,0)$ or $(1,0,1)$. Since we are assuming that $(x_1,x_2,x_3)$ is 
obtained by a sequence of mutations and sign changes from either 
$(-1,-1,0)$ or $(-1,0,-1)$, this implies that the disjoint union~\eqref{e:tripleofballs-1-1-1} 
smoothly embeds in $\CP^2$, which is what we needed to show.
This covers the embeddings of Theorem~\ref{t:disj-emb}(2).

\subsection{$\bde = (-1,1,1)$, $a=4$ and $m_L=-\tau_\la^{-5}$}\label{ss:pairs=-111-a=4}

We observed in Section~\ref{ss:hurwitz} and Remark~\ref{r:factperm} that for each cyclic permutation $\sigma$ of $(1,2,3)$ the transformation 
$(\bde,\bx)\mapsto\si(\bde,\bx)$ is realized by a composition of two Hurwitz moves on the factorization. 
Therefore, the permutation $(123)$ induces a bijection $S^{-1,1,1}_4\to S^{1,-1,1}_4$ which preserves the set of integer triples realized as intersection numbers of triples of curves. 
By the results of Section~\ref{ss:-111-a=4} it follows that a triple 
$(x_1,x_2,x_3) \in S_4^{1,-1,1}$ is realized by three curves 
$\ga_1,\ga_2,\ga_3$ if and only if it is obtained from $(1,2,1)$ via mutations and sign changes.

It turns out that such triples, as well as the corresponding curves $\ga_i$, admit an explicit description involving the Fibonacci and the Lucas numbers. 
Recall that the Lucas sequence $(L_n)_{n \in \Z}$ is defined by the same recursive relation as the Fibonacci sequence, but starting from $L_0=2$ and $L_1=1$. We will use the following identities. 
\begin{align}\label{e:lucas} 
& L_n = F_{n-1}+F_{n+1}\tag{I1} \\
& L_{-n} = (-1)^n L_n\quad\text{and}\quad F_{-n}=(-1)^{n+1}F_n\tag{I2} \\
& L_m L_n = L_{m+n}+(-1)^n L_{m-n}\tag{I3} \\
& 5F_m F_n=L_{m+n}-(-1)^n L_{m-n}\tag{I4} \\
& L_mF_n=F_{m+n}-(-1)^n F_{m-n}\tag{I5} \\
& L_n^2=5F_n^2+4(-1)^n=L_{2n}+2(-1)^n\tag{I6}\\
& \gcd(F_m,F_n) = F_{\gcd(m,n)}\tag{I7}\\
& F_{m+n} = F_{m+1} F_n + F_m F_{n-1}\tag{I8}\\
& F_n^2 = F_{n+r} F_{n-r} + (-1)^{n-r} F_r^2,\quad n\geq r>0\tag{I9}
\end{align}
All of the above identities can be found e.g.~on the Wikipedia pages~\cite{Wifib, Wiluc}, except perhaps (I3), which follows easily from (I1) and (I5). 
Proposition~\ref{p:lucas-triples} below shows that the triples $(x_1,x_2,x_3)$ we 
are interested in have a specific form, described by the following definition. 
\begin{defn}
Let $a,b\in\Z$ be odd and coprime and $\ep \in \{\pm 1\}$ and denote by $L(\ep,a,b)$ 
the triple $(L_a,-\ep L_{a+\ep b},L_b)$. Define $L\subset\Z^3$ to be the subset of triples 
of the form $L(\ep,a,b)$ with $\ep,a,b$ as above. 
\end{defn}
\begin{prop}\label{p:lucas-triples}
The set $L$ is contained in $S_4^{1,-1,1}$ and is invariant under mutations and sign changes. 
More specifically, let $\mu_i$ be the $i$-th mutation map defined in Section~\ref{ss:hurwitz} and $\bde = (1,-1,1)$. Then, the following formulas hold:
\[
\mu_i(\bde, L(\ep,a,b))=
\begin{cases}
(\bde, L(\ep,-a-2\ep b,b)) & \text{if $i=1$,} \\
(\bde, L(-\ep,a,b))\quad & \text{if $i=2$},\\
(\bde, L(\ep,a,-b-2\ep a))\quad & \text{if $i=3$.} 
\end{cases}
\]
Moreover, $L$ coincides with the set of triples obtained from $(1,2,1)$ by mutations and sign changes.
\end{prop}

\begin{proof}
By looking at Equation~\eqref{e:negcohn} for $\bde = (1,-1,1)$ and $a=4$,  
it is easy to check that, given odd coprime integers $a$ and $b$, the triples 
$L(1,a,b)$ and $L(-1,a,b)$ are in $S_4^{1,-1,1}$ if and only if 
$-L_{a+b}$ and $L_{a-b}$ are the roots of the polynomial 
\[
x^2+L_a L_b x-(L_a^2+L_b^2+4),
\]
or equivalently that
\[
-L_{a+b}+L_{a-b}=-L_aL_b \quad \text{and} \quad L_{a+b}L_{a-b}=L_a^2+L_b^2+4. 
\]
The first equality follows from (I2) and (I3), the second one from (I3) 
and the identity $L_a^2+L_b^2+4=L_{2a}+L_{2b}$, which follows from (I6). 
Now we compute the mutations of a generic triple 
$(x_1,x_2,x_3)=(L_a,-\ep L_{a+ \ep b},L_b)$. The argument just provided shows that 
$\hat x_2=\ep L_{a-\ep b}$, therefore $\mu_2(\bde, L(\ep,a,b))=(\bde, L(-\ep,a,b))$. 
Moreover, using (I2) and (I3), 
\[
\hat x_1=-\ep L_{a+\ep b}L_b-L_a=-\ep L_{\ep a + b}L_b-L_a=-\ep(L_{\ep a+2b}-L_{\ep a})-L_a=L_{-a-2\ep b}
\]
so that 
\[
\mu_1(\bde, L(\ep,a,b))=(\bde, L(\ep,-a-2\ep b,b)).
\]
A similar computation yields $\hat x_3=L_{-b-2\ep a}$ and 
\[
\mu_3(\bde, L(\ep,a,b))=(\bde, L(\ep,a,-b-2\ep a)).
\]
Finally, the triples that obtained from $L(\ep,a,b)$ by sign changes are
\begin{itemize}
    \item $(-L_a,\ep L_{a+\ep b},L_b)=(L_{-a},\ep L_{-a-\ep b},L_b)=L(-\ep,-a,b)$;
    \item $(-L_a,-\ep L_{a+\ep b},-L_b)=(L_{-a},-\ep L_{-a-\ep b},L_{-b})=L(\ep,-a,-b)$;
    \item $(L_a,\ep L_{a+\ep b},-L_b)=(L_a,\ep L_{-a-\ep b},L_{-b})=L(-\ep,a,-b)$.
\end{itemize}
This proves the first part of the statement. In order to prove the last sentence we observe that $(1,2,1)=L(-1,1,1)$. 
Combining this with the first part of the statement we get that every triple obtained from $(1,2,1)$ by mutations and sign changes 
has the form $L(\ep,a,b)$. Conversely, we claim that any triple $L(\ep,a,b)$ can be transformed by a 
(possibly empty) sequence of mutations into $L(\ep',a',b')$ with $|a'|=|b'|$. To prove the 
claim observe that if $|a| \neq |b|$ then the elements of the triple have pairwise distinct absolute 
values. Therefore, by Proposition~\ref{p:negcohn} there exists a mutation which makes 
the sum of the absolute values strictly decrease, and the claim follows from a standard 
infinite descent argument. Since $a'$ and $b'$ are coprime, we must have $|a'|=|b'|=1$, 
and up to a change of signs we can assume $a'=b'=1$, so that $L(\ep',a',b')$ is 
either $(1,2,1)$ or $(1,-3,1)$, which are mutants of each other.
\end{proof}
Consider a triple $(x_1,x_2,x_3)=(L_a,-\ep L_{a+ \ep b},L_b) \in S_4^{1,-1,1}$. 
Then, Equations~\eqref{e:xp-equations} take the form:  
\begin{equation}
\begin{cases}
4 + x_1^2 = 5 p_1^2\\
2x_3 + x_1 x_2 + x_1^2 x_3 = 5 p_1 p_2\\
2x_2 + x_1 x_3 = 5 p_1 p_3.
\end{cases}
\end{equation}
Using Identity (I6) we get $p_1=\pm F_a$ from the first equation, 
while using Identity (I4) the third equation can be written as 
\[
5p_1p_3=x_2-\hat x_2=-\ep L_{a+\ep b}-\ep L_{a-\ep b}=
-\ep(L_{a+b}+L_{a-b})=-5 \ep F_a F_b
\]
so that $p_3=\mp \ep F_b$ (since $F_a \neq 0$). Finally, using (I2), (I3) and 
(I4) we can write the second equation as
\[
\begin{split} 
5p_1p_2=2x_3+x_1(x_2+x_1x_3)=2x_3-x_1 \hat x_2=2L_b-\ep L_a L_{a-\ep b}=\\
=2L_b-\ep (L_{2a-\ep b}+L_{\ep b})=L_b-L_{2 \ep a-b}=5F_aF_{b-\ep a}
\end{split} 
\]
so that $p_2=\pm F_{b-\ep a}$. Therefore we have
\begin{equation}\label{e:gammas}
\ga_1 = \pm F_a\mu + q_1\la,\quad 
\ga_2 = \pm F_{b-\ep a}\mu + q_2\la\quad\text{and}\quad
\ga_3 = \mp \ep F_b \mu + q_3\la  
\end{equation} 
and, due to Relations~\eqref{e:relations1}, the $q_i$'s satisfy 
\begin{equation}\label{e:qsystem3}
F_a q_2 - F_{b-\ep a} q_1 =  \pm L_b,\quad 
F_a q_3 + \ep F_b q_1 = \mp \ep L_{a+\ep b}\quad\text{and}\quad 
F_{b-\ep a} q_3 + \ep F_b q_2 = \pm L_a.
\end{equation} 
Observe that $F_{b-\ep a} = 0$ implies $b=\ep a=\pm 1$ and therefore $F_a=F_b=1$. 
This leads to a case which falls outside the scope of Theorem~\ref{t:disj-emb}, therefore from now on we assume $F_{b-\ep a} \neq 0$. 
Note that the first and last of Equations~\eqref{e:qsystem3} 
imply 
\begin{equation}\label{e:qcong3}
q_1\equiv\mp\frac{L_b}{F_{b-\ep a}}\bmod F_a,\quad 
q_2\equiv\pm\frac{L_b}{F_a}\bmod F_{b-\ep a},\quad\text{and}\quad  
q_3\equiv\pm\frac{L_a}{F_{b-\ep a}}\bmod F_b.
\end{equation}

We conclude that if a $(1,1,3,1)$-decomposition with monodromy $-\tau_\la^{-5}$ 
induces an orientation-preserving embedding in $\CP^2$ of a 
disjoint union of $\pm B_{x,q}$'s, then such a union is of the form 
\begin{equation}\label{e:tripleofballs1-11}
B_{F_a,q_1}\cup -B_{F_{b-\ep a},q_2}\cup B_{F_b,q_3}
\end{equation}
for some odd, coprime integers $a,b$ and some $\ep\in\{\pm 1\}$, and where the $q_i$'s 
satisfy~\eqref{e:qcong3}. 

Now we want to argue that, conversely, if $a,b$ are odd and coprime and $\ep\in\{\pm 1\}$, then the disjoint union~\eqref{e:tripleofballs1-11}, where the $q_i$'s 
satisfy~\eqref{e:qcong3}, smoothly embeds in $\CP^2$. 
We start observing that, given $a,b$ and $\ep$ as above, 
there is a triple $(q_1,q_2,q_3)$ satisfying 
Equations~\eqref{e:qsystem3} with the first choice of sign on the right-hand side. 
Indeed, by (I2) and (I5) we have
\[
\ep L_b F_b - F_a L_a = F_{2\ep b} - F_{2a} = 
- L_{a+\ep b} F_{a - \ep b} = \ep L_{a+\ep b} F_{b-\ep a} \equiv 0\bmod F_{b-\ep a},
\]
therefore $L_b/F_a\equiv L_a/\ep F_b\bmod F_{b-\ep a}$. This guarantees that we can find 
integers $q_1$, $q_2$ and $q_3$ satisfying the first and third Equations~\eqref{e:qsystem3}.
On the other hand, $F_a$ times the third equation minus $\ep F_b$ times the first 
equation gives 
\[
F_{b-\ep a}(F_a q_3 + \ep F_b q_1) = F_a L_a - \ep F_b L_b 
\stackrel{{\rm (I5)}}{=} F_{2a} - F_{2\ep b} = -\ep F_{b-\ep a} L_{a+\ep b},
\]
and therefore the second equation holds.  
Now define $\ga_i$, for $i=1,2,3$, as in~\eqref{e:gammas} with the first choice of sign 
and consider the cobordism $X$ with a horizontal decomposition of type $(1,1,3,1)$ 
prescribed by the factorization $(\tau_{\ga_3},\tau_{\ga_2}^{-1},\tau_{\ga_1})$.
We claim that Equation~\eqref{e:eq2} holds for $\bde = (1,-1,1)$ replacing $(p_1,p_2,p_3)$ with 
$(F_a,F_{b-\ep a},-\ep F_b)$ and $(x_1,x_2,x_3)$ with $(L_a,-\ep L_{a+\ep b}, L_b)$. 
Indeed, after these substitutions Equation~\eqref{e:eq2} is equivalent to  
\begin{equation}\label{e:fibolucas}
F_a^2 - F_{b-\ep a}^2 + F_b^2 + F_a F_{b-\ep a} L_b - F_a F_b L_{a+\ep b} 
- \ep F_{b-\ep a} F_b L_a + \ep F_a F_b L_a L_b = 0,    
\end{equation}
which turns out to be an identity involving Fibonacci and Lucas numbers. To see 
why~\eqref{e:fibolucas} holds, use (I2) and (I3) to replace the last term on the 
left-hand side with $F_a F_b (L_{a+\ep b} - L_{a-\ep b})$, $F_a F_{b-\ep a} L_b$ 
with $F_a (F_{2b-\ep a}-F_a)$, $F_a F_b L_{a+\ep b}$ with $F_a (F_{2b-\ep a} + F_a)$
and $-F_{b-\ep a} F_b L_{\ep a}$ with $-F_b^2 + F_{2a-\ep b} F_b$. 
After cancellations Equation~\eqref{e:fibolucas} becomes
\[
F_{2a-\ep b} F_b - F_a^2 - F_{a-\ep b}^2 = 0,
\]
which holds by (I9). 
Since Equation~\eqref{e:eq2} is equivalent to $m_L(\la)\cdot\la=0$, 
this implies $m_L(\la) = \pm \la$ and therefore $X\co S^3\to S^3$. 
Moreover, by Equations~\eqref{e:qsystem3} the triple 
$(\ga_2\cdot\ga_3, \ga_1\cdot\ga_3,\ga_1\cdot\ga_2) = (L_a,-\ep L_{a+\ep b},L_b)$ 
belongs to $L\subset S^{1,-1,1}_4$, therefore by the last statement of Proposition~\ref{p:lucas-triples}
it is obtained from $(1,2,1)$ by mutations and sign changes. Since the identification 
of $S^{1,-1,1}_4$ with $S^{-1,1,1}_4$ is induced by Hurwitz moves, by the analysis of 
Section~\ref{ss:-111-a=4} we conclude that $\hat X\cong\CP^2$ and the 
disjoint union~\eqref{e:tripleofballs1-11} smoothly embeds in it. 
This covers the embeddings of Theorem~\ref{t:disj-emb}(3) and concludes the proof 
of Theorem~\ref{t:disj-emb}.

\section{Proof of Theorem~\ref{t:manyballs}}\label{s:manyballs} 

The organization of the proof requires that we start with Theorem~\ref{t:manyballs}(3). 
We need the following result. 

\begin{thm}[Carmichael's theorem~\cite{Ca13}]\label{t:carmichael}
Let $n$ be a positive integer, $n \notin \{1,2,6,12\}$. Then $F_n$ admits a primitive factor, 
i.e. a prime $p$ such that $n=\min \{k \in \N^+ \; | \; p \text{ divides } F_k \}$. 
\end{thm}

\begin{rmk}\label{r:carm-cor}
An immediate consequence of Carmichael's theorem is that if $n$ is odd and $n>3$, then $F_n$ 
always admits an {\em odd} primitive factor $p$. Moreover, by (I7), for any integer $k$ we have 
\[
p\ |\ F_k\quad \Longleftrightarrow\quad p\ |\ \gcd(F_n,F_k) = F_{\gcd(n,k)} 
\quad\Longleftrightarrow\quad n\ |\ k.
\]
\end{rmk}

The following statement implies Theorem~\ref{t:manyballs}(3). 
We shall prove it after a few preliminaries. 

\begin{lemma}\label{l:precise}
Given an odd integer $a>1$, let $S(a)\subset\F_3$ be the subset of elements of the form $B_{F_a,q}$, and let $\psi\co\N_{>1}\to\N$ be the map given by $\psi(a) = \left|S(a)\right|$. 
Then, for each odd prime $a>3$ we have $\psi(a)= (a-1)/2$. In particular, $\psi$ is unbounded. Moreover, for each odd 
integer $a>3$ the map $f\co S(a)\to\Z/F_a\Z$ 
given by $f(B_{F_a,q}) = q^2\bmod F_a$ is injective. 
\end{lemma}

Let $a>1$ be an odd integer. 
Given $b\in\Z$ odd and coprime with $a$ and $\ep\in\{\pm 1\}$, by Theorem~\ref{t:disj-emb}(3) there is a smooth, orientation-preserving embedding $B_{F_a,q}\subset\CP^2$, where $q\equiv \pm L_b/F_{b-\ep a}\bmod F_a$. 
We would like to understand how many different such embeddings we can construct if we keep $a$ fixed. 
In other words, setting $I(b,\ep):=L_b/F_{b-\ep a} \bmod F_a$, we want to understand for which pairs $(b,\ep)$ and $(b',\ep')$ we have $I(b,\ep)\equiv \pm I(b',\ep')\bmod F_a$.
We are going to do this assuming $a>3$, because the only $B_{p,q}$'s that we have when $a=1$ or $a=3$ are, respectively, $B_{1,0}=B^4$ and $B_{2,1}$, which are both already known to embed in $\CP^2$.
In principle  we should allow $b$ and $\ep$ to vary independently of each other, but we observe that, since $a$ and $b$ are both odd, by 
(I2) we have 
\[
I(b,\ep) = L_b/F_{b-\ep a} = (-L_b)/(-F_{b-\ep a}) = L_{-b}/F_{-b-(-\ep)a} = I(-b,-\ep).
\]
Therefore, it will suffice to fix $\ep=-1$ and understand for which $b,b'$ we have $I(b,-1)\equiv\pm I(b',-1)\bmod F_a$.

\begin{prop}\label{p:I(b,ep)}
Let $a,b$ odd, coprime integers with $a>3$. Then, 
$I(b,-1)\equiv\pm I(b',-1)\bmod F_a$ if and only if $b \equiv \pm b'\bmod a$. 
\end{prop} 

\begin{proof} 
Using (I5) we see that the congruences $I(b,-1)\equiv\pm I(b',-1)\bmod F_a$ are equivalent to 
\[
F_{b+b'+a} - F_{b-b'-a} = L_b F_{b'+a} \equiv 
\pm L_{b'} F_{b+a} = L_{\pm b'}F_{b+a} =F_{\pm b'+b+a} - F_{\pm b'-b-a}\bmod F_a
\]
and using (I8) to 
\begin{equation}\label{e:FLcong}
\begin{split}
(F_{b+b'} - F_{\pm b'+b}) F_{a-1} 
+  (F_{\pm b'-b} - F_{b-b'}) F_{-a -1}
\equiv 
0\bmod F_a.
\end{split}
\end{equation} 
Since $a$ is odd, by (I2) we have $F_{-a-1}=-F_{a+1} \equiv -F_{a-1} \bmod F_a$,  
and since $F_{a-1}$ and $F_a$ are coprime Equation~\eqref{e:FLcong} is equivalent to 
\[
(F_{b+b'} - F_{\pm b'+b}) 
-  (F_{\pm b'-b} - F_{b-b'}) \equiv 0 \bmod F_a
\]
Thus, we have 
\begin{equation}\label{e:simpcong}
\begin{cases}
2F_{b-b'} \equiv 0 \bmod F_a & \quad \text{if and only if $I(b,-1)\equiv I(b',-1)\bmod F_a$}\\
2F_{b+b'} \equiv 0 \bmod F_a & \quad \text{if and only if $I(b,-1)\equiv -I(b',-1)\bmod F_a$.}
\end{cases}
\end{equation} 
Now observe that if $b \equiv \pm b' \bmod a$ then by (I7) 
$\gcd(F_{b\mp b'},F_a) = F_a$, therefore by Equations~\eqref{e:simpcong} 
we have $I(b,-1)\equiv \pm I(b',-1)\bmod F_a$. Conversely, by 
Remark~\ref{r:carm-cor} the number $F_a$ admits an odd primitive factor $p$. 
Equations~\eqref{e:simpcong} imply that $p$ must divide $F_{b \mp b'}$ if 
$I(b,-1)\equiv \pm I(b',-1)\bmod F_a$, and again by Remark~\ref{r:carm-cor}
this is equivalent to $a\ |\ b \mp b'$. In other words, $b \equiv \pm b' \bmod a$.
\end{proof} 

\begin{proof}[Proof of Lemma~\ref{l:precise}]  
Given any odd integer $b$ coprime with $a$ 
we have $I(b,-1)=I(b_0,-1)$ for some odd $b_0$ such that $b_0 \equiv b \bmod 2a$ 
and $0<b_0<2a$. On the other hand, given two odd integers $b$ and $b'$, both coprime 
with $a$ and both lying in the interval $(0,2a)$, we have
\[ 
b \equiv b' \bmod a\ \Longleftrightarrow\  b=b' 
\qquad \text{and} \qquad 
b \equiv -b' \bmod a\ \Longleftrightarrow\ b+b'=2a.
\]
Since an odd integer is coprime with $a$ if and only if it is coprime with $2a$, the 
interval $(0,2a)$ contains exactly $\varphi(2a)$ odd integers coprime with $a$, 
where $\varphi$ denotes Euler's totient function, and since the involution 
$x \mapsto 2a-x$ has no fixed points among such integers, we 
get $\psi(a)=\dfrac{\varphi(2a)}{2}=\dfrac{\varphi(a)}{2}$, 
which takes the value $(a-1)/2$ when $a$ is prime. In particular, $\psi$ is unbounded. 
This establishes the first part of the statement. 

Now observe that 
$I(b,-1)^2=L_b^2/F_{b+a}^2 = L_{b'}^2/F_{b'+a}^2=I(b',-1)^2\bmod F_a$ if and only if 
\begin{equation}\label{e:qsquare}
L_{b'}^2 F_{b+a}^2 - L_b^2 F_{b'+a}^2  = 
(L_{b'} F_{b+a} + L_b F_{b'+a})(L_{b'} F_{b+a} - L_b F_{b'+a})\equiv 0\bmod F_a.
\end{equation} 
Moreover, by (I5) 
\[
L_{b'} F_{b+a} - L_b F_{b'+a} =  
F_{b-b'-a} - F_{b'-b-a} = L_{-a} F_{b-b'} 
\equiv 2 F_{-a - 1} F_{b-b'} 
\]
and by (I2), (I5) and (I8) 
\[
\begin{split}
L_{b'} F_{b+a} + L_b F_{b'+a} = 
2F_{b'+b+a} - F_{b'-b-a} - F_{b-b'-a} \equiv    
2F_{-a-b-b'} - F_{-a + 1} (F_{b'-b}+F_{b-b'})\equiv \\
\equiv 2 F_{-a + 1} F_{-b-b'} \bmod F_a. 
\end{split}
\]
Therefore Equation~\eqref{e:qsquare} is equivalent to 
\[
4 F_{b-b'} F_{b+b'} \equiv 0\bmod F_a.
\]
Let $p$ be a primitive factor of $F_a$. As observed in Remark~\ref{r:carm-cor} $p$ is necessarily odd. 
Then, $p$ divides at least one among $F_{b-b'}$ and $F_{b+b'}$, 
and by Remark~\ref{r:carm-cor} we conclude that $b \equiv \pm b' \bmod a$ 
and therefore by Proposition~\ref{p:I(b,ep)} $I(b,-1)=\pm I(b',-1)\bmod F_a$. 
This concludes the proof of the lemma.
\end{proof} 

By Lemma~\ref{l:precise} the sequence $\{\psi(a)\}_a$ is unbounded. 
Since, using the notation of Theorem~\ref{t:manyballs}, we have $|S_{F_a}|\geq\psi(a)$, this implies Theorem~\ref{t:manyballs}(3). 
We are now going to prove Theorem~\ref{t:manyballs}(1) and (2). 
Suppose $B_{p,q}\in\F_1$. Without loss of generality we may assume $B_{p,q}=B_{p_1,q_1}$, with $(p_1,p_2,p_3)$ satisfying Markov's equation $p_1^2+p_2^2+p_3^2 = 3 p_1 p_2 p_3$. 
This implies $p_2^2+p_3^2\equiv 0\bmod p_1$, therefore $q^2 = q_1^2 \equiv 9p_2^2/p_3^2 \equiv -9\bmod p_1$. 
Similarly, if $B_{p,q}\in\F_2$ then $B_{p,q}=B_{x_1,q_1}$ and using the equation $x_1^2+ x_1 x_2 x_3 = x_2^2+x_3^2$ and the congruence $x_2 q_1 \equiv\pm x_3\bmod x_1$ one concludes $q^2\equiv -1\bmod p$. 
This proves one direction of both Theorem~\ref{t:manyballs}(1) and (2). 

To prove the other direction of Theorem~\ref{t:manyballs}(1) it suffices to show that any rational ball $B_{p,q}$ appearing in the statement of Theorem~\ref{t:disj-emb} and satisfying $q^2\equiv -9\bmod p$ actually belongs to $\F_1$. 
If $B_{p,q}\in\F_2$, $q^2\equiv -1\bmod p$ implies $p\ |\ 8$ and we are forced to have $p \in \{1,2\}$ because $q^2\equiv -1\bmod p$ is impossible if $4\ |\ p$. 
Therefore $B_{p,q} \in \{ B_{1,0},B_{2,1} \}$. But $B_{1,0}$ and $B_{2,1}$ belong to $\F_1$ because they both arise from the Markov triple $(2,1,1)$. 
Now suppose that $B_{F_a,q}\in\F_3$. When $a = 3$ the rational ball $B_{F_a,q}$ is equal to $B_{2,1}$ and 
therefore belongs $\F_1$. Thus, we may assume $a>3$. We claim 
that $q^2\equiv -9\bmod F_a$ implies $B_{F_a,q} = B_{F_a,F_{a-4}}$. 
Indeed, observe that, choosing $b=2-a$ and $\ep=-1$ we have 
$q\equiv\pm L_{2-a}/F_2\equiv\pm L_{2-a}\bmod F_a$. Moreover, 
\[
L_{2-a} = -L_{a-2} = - (F_{a-3} + F_{a-1}) = - (F_a - F_{a-4}) \equiv F_{a-4}\bmod F_a. 
\]
Therefore, by (I9) we have $q^2\equiv F_{a-4}^2 \equiv -F_4^2 \equiv -9 \bmod F_a$. 
In view of the last statement of Lemma~\ref{l:precise}, this proves the claim. Since 
$(1,F_{a-2},F_a)$ is a Markov triple and $F_{a-4}\equiv 3F_{a - 2}\bmod F_a$, 
we have $B_{F_a,F_{a-4}}\in\F_1$, and the other direction of Theorem~\ref{t:manyballs}(1) holds. 

To prove the other direction of Theorem~\ref{t:manyballs}(2), as in the previous case it suffices to show that 
any rational ball $B_{p,q}$ appearing in the statement of Theorem~\ref{t:disj-emb} and 
satisfying $q^2\equiv -1\bmod p$, actually belongs to $\F_2$. 
As before, If $B_{p,q}\in\F_1$ then $B_{p,q} \in \{B_{1,0},B_{2,1}\}$, and both $B_{1,0}$ and $B_{2,1}$ belong to $\F_2$ 
because they arise, respectively, from the triples $(x_1,x_2,x_3) = (1,1,0)$ and $(x_1,x_2,x_3) = (2,1,3)$. If $B_{F_a,q}\in\F_3$, 
as before we may assume $a>3$. 
We claim that $q^2\equiv -1\bmod F_a$ implies $B_{F_a,q} = B_{F_a,F_{a-2}}$. 
This time we choose $b=1$ and $\ep=-1$. 
Then, by (I9) $q\equiv\pm L_1/F_{1+a}\equiv\mp F_{a-1}\equiv\pm F_{a-2} \bmod F_a$ and 
\[
F_{a-2}^2 = F_a F_{a-4} - F_2^2 \equiv -1\bmod F_a. 
\]
Therefore $q^2\equiv -1\bmod F_a$ implies $q^2 \equiv F_{a-2}^2\bmod F_a$. 
As before, in view of Lemma~\ref{l:precise} this proves the claim. 
Since by (I9) the triple $(x_1,x_2,x_3) = (F_a,F_{a+1},1)$ solves the 
equation $x_1^2+ x_1 x_2 x_3 = x_2^2+x_3^2$, 
\[
F_{a+1} F_{a-1} \equiv -1\bmod F_a\quad\text{and}\quad 
F_{a-2}\equiv -F_{a-1}\bmod F_a,
\]
we have $B_{F_a,F_{a-2}}\in\F_2$ and the 
other direction of Theorem~\ref{t:manyballs}(2) holds. This concludes the proof of Theorem~\ref{t:manyballs}.

\printbibliography
\Addresses
\end{document}